%% file: main.tex
\pdfoutput=1
\documentclass[11pt, letterpaper]{article}

\input{macros}

\title{Separating Oblivious and Adaptive Models of Variable Selection}
\date{}

\author{Ziyun Chen\thanks{
University of Washington.
\texttt{ziyuncc@cs.washington.edu}.
} \and Jerry Li\thanks{
University of Washington.
\texttt{jerryzli@cs.washington.edu}.
} \and Kevin Tian\thanks{
University of Texas at Austin.
\texttt{kjtian@cs.utexas.edu}.
}\and Yusong Zhu\thanks{
University of Texas at Austin.
\texttt{zhuys@utexas.edu}.
}}
\begin{document}

\maketitle

\begin{abstract}
Sparse recovery is among the most well-studied problems in learning theory and high-dimensional statistics. In this work, we investigate the statistical and computational landscapes of sparse recovery with $\ell_\infty$ error guarantees. This variant of the problem is motivated by \emph{variable selection} tasks, where the goal is to estimate the support of a $k$-sparse signal in $\R^d$. Our main contribution is a provable separation between the \emph{oblivious} (``for each'') and \emph{adaptive} (``for all'') models of $\ell_\infty$ sparse recovery. We show that under an oblivious model, the optimal $\ell_\infty$ error is attainable in near-linear time with $\approx k\log d$ samples, whereas in an adaptive model, $\gtrsim k^2$ samples are necessary for any algorithm to achieve this bound. This establishes a surprising contrast with the standard $\ell_2$ setting, where $\approx k \log d$ samples suffice even for adaptive sparse recovery. We conclude with a preliminary examination of a \emph{partially-adaptive} model, where we show nontrivial variable selection guarantees are possible with $\approx k\log d$ measurements.\footnote{Extended abstract accepted for presentation at the Conference on Learning Theory (COLT) 2026}
\end{abstract}
\thispagestyle{empty}
\newpage
\tableofcontents
\thispagestyle{empty}
\newpage

\input{tex_file/intro}
\input{tex_file/preliminaries}
\input{tex_file/model1}
\input{tex_file/model3}
\input{tex_file/model2}

\section*{Acknowledgements}
We are grateful to the NSF AI Institute for Foundations of Machine Learning (IFML) for their support of this project.
YZ and KT thank Eric Price for helpful conversations on the $\ell_\infty$ sparse recovery literature and the partially-adaptive model in Section~\ref{sec:for_all}.
\newpage
\bibliography{ref}
\bibliographystyle{alpha}

\newpage 
\appendix
\input{tex_file/append}
\input{tex_file/appendix2}
\end{document}

%% file: macros.tex
\usepackage[english]{babel}
\usepackage[T1]{fontenc}
\usepackage{parskip}
\usepackage[margin=1in]{geometry}
\usepackage{bbm,bm}
\usepackage{bbold}
\usepackage{parskip}
\usepackage{amsmath}
\usepackage{amssymb}
\usepackage{amsthm}
\usepackage{thmtools, thm-restate}
\usepackage{booktabs}
\usepackage{array}
\usepackage[lined,boxed,ruled,norelsize,linesnumbered]{algorithm2e}
\usepackage{graphicx} 
\usepackage{enumitem}

\usepackage{hyperref}
\hypersetup{
	colorlinks=true,
	linkcolor=blue!70!black,
	citecolor=blue!70!black,
	urlcolor=blue!70!black
}
\usepackage{graphicx}
\usepackage[english]{babel}
\usepackage[utf8x]{inputenc}
\usepackage[T1]{fontenc}
\usepackage{mathtools}
\usepackage{cleveref}
\usepackage{parskip}
\usepackage{xcolor}
\newtheorem{theorem}{Theorem}
\newtheorem{lemma}{Lemma}
\newtheorem{fact}{Fact}
\newtheorem{definition}{Definition}
\newtheorem{corollary}{Corollary}
\newtheorem{proposition}{Proposition}
\newtheorem{model}{Model}

\newtheorem{problem}{Problem}

\newtheorem{remark}{Remark}

\newcommand{\defeq}{:=}

\newcommand{\norm}[1]{\left\lVert#1\right\rVert}
\newcommand{\norms}[1]{\lVert#1\rVert}
\newcommand{\normop}[1]{\left\lVert#1\right\rVert_{\textup{op}}}
\newcommand{\normf}[1]{\left\lVert#1\right\rVert_{\textup{F}}}

\newcommand{\inprod}[2]{\left\langle#1, #2\right\rangle}

\newcommand{\eps}{\epsilon}
\newcommand{\lam}{\lambda}
\newcommand{\R}{\mathbb{R}}

\newcommand{\N}{\mathbb{N}}

\newcommand{\half}{\frac{1}{2}}

\newcommand{\E}{\mathbb{E}}

\newcommand{\Var}{\textup{Var}}

\newcommand{\Nor}{\mathcal{N}}

\newcommand{\Tr}{\textup{Tr}}

\newcommand{\id}{\mathbf{I}}

\newcommand{\polylog}{\textup{polylog}}
\newcommand{\Par}[1]{\left(#1\right)}
\newcommand{\Brack}[1]{\left[#1\right]}
\newcommand{\Brace}[1]{\left\{#1\right\}}
\newcommand{\Abs}[1]{\left|#1\right|}

\newcommand{\ind}{\mathbb{I}}

\newcommand{\nnz}{\mathrm{nnz}}
\newcommand{\poly}{\mathrm{poly}}

\newcommand{\RIP}{\mathrm{RIP}}

\newcommand{\supp}{\mathrm{supp}}

\newcommand{\vth}{\boldsymbol{\theta}}

\newcommand{\matp}{\mathbf{P}}

\newcommand{\vxi}{\boldsymbol{\xi}}
\newcommand{\sig}{\sigma}

\newcommand{\twonorm}[1]{\norm{#1}_2}
\newcommand{\onenorm}[1]{\norm{#1}_1}
\newcommand{\infnorm}[1]{\norm{#1}_\infty}

\newcommand{\vlam}{\boldsymbol{\lambda}}
\newcommand{\vzero}{\mathbf{0}}
\newcommand{\vone}{\mathbf{1}}
\newcommand{\vths}{\vth^\star}
\newcommand{\vhth}{\widehat{\vth}}
\newcommand{\vthp}{\vth^\prime}

\newcommand{\kldiv}[1]{D_{\textup{KL}}\Par{#1}}

\newcommand{\cov}[1]{\mathrm{Cov}\left[#1\right]}

\newcommand{\expect}[1]{\mathbb{E}\left[#1\right]}

\newcommand{\Prob}{\mathbb{P}}
\renewcommand{\Pr}{\Prob}
\newcommand{\prob}[1]{\mathbb{P}\left[#1\right]}

\newcommand{\ma}{{\mathbf{A}}}

\newcommand{\mc}{{\mathbf{C}}}

\newcommand{\me}{{\mathbf{E}}}

\newcommand{\mg}{{\mathbf{G}}}

\newcommand{\mi}{{\mathbf{I}}}

\newcommand{\mm}{{\mathbf{M}}}

\newcommand{\mr}{{\mathbf{R}}}

\newcommand{\mx}{{\mathbf{X}}}

\newcommand{\mX}{{\mathbf{X}}}

\newcommand{\vb}{{\mathbf{b}}}

\newcommand{\ve}{{\mathbf{e}}}

\newcommand{\vq}{{\mathbf{q}}}
\newcommand{\vr}{{\mathbf{r}}}

\newcommand{\vu}{{\mathbf{u}}}
\newcommand{\vv}{{\mathbf{v}}}
\newcommand{\vw}{{\mathbf{w}}}
\newcommand{\vx}{{\mathbf{x}}}
\newcommand{\vy}{{\mathbf{y}}}

\newcommand{\calA}{{\mathcal{A}}}

\newcommand{\calE}{{\mathcal{E}}}
\newcommand{\calF}{{\mathcal{F}}}

\newcommand{\calM}{{\mathcal{M}}}
\newcommand{\calN}{{\mathcal{N}}}

\newcommand{\calP}{{\mathcal{P}}}

\newcommand{\calV}{{\mathcal{V}}}

\newcommand{\calX}{{\mathcal{X}}}

\newcommand{\bbE}{{\mathbb{E}}}

\newcommand{\bbP}{{\mathbb{P}}}

\newcommand{\bbR}{{\mathbb{R}}}

\newcommand{\stepa}[1]{\overset{(a)}{#1}}
\newcommand{\stepb}[1]{\overset{(b)}{#1}}
\newcommand{\stepc}[1]{\overset{(c)}{#1}}
\newcommand{\stepd}[1]{\overset{(d)}{#1}}
\newcommand{\stepe}[1]{\overset{(e)}{#1}}

\newcommand{\paren}[1]{\left(#1\right)}
\newcommand{\bra}[1]{\left[#1\right]}
\newcommand{\cbra}[1]{\left\{#1\right\}}

\crefname{assumption}{assumption}{assumptions}

\newcommand{\innp}[1]{\left \langle #1 \right\rangle}

\newcommand{\sign}{\mathrm{sign}}

\definecolor{burntorange}{rgb}{0.8, 0.33, 0.0}
\definecolor{skyblue}{HTML}{75aadb}

\newcommand{\xxinv}{[\mx^\top\mx]^{-1}}
\newcommand{\ktwonorm}[1]{\norm{#1}_{2, k}}

\newcommand{\subg}{\textup{subG}}
\newcommand{\infbound}{\infnorm{\mx^\top \vxi}}

\newcommand{\sfp}{{S_{\textup{FP}}}}
\newcommand{\sfn}{{S_{\textup{FN}}}}
\newcommand{\stp}{{S_{\textup{TP}}}}

\newcommand{\minmax}{\textup{minimax}}
\newcommand{\bayes}{\textup{Bayes}}
\newcommand{\IHT}{\mathsf{IHT}}
\newcommand{\OSR}{\mathsf{ObliviousSparseRecovery}}
\newcommand{\OSRR}{\mathsf{OSRReduction}}
\newcommand{\bvth}{\bar{\vth}}
\newcommand{\AO}{\mathsf{AdaptiveObserve}}
\newcommand{\AOSR}{\mathsf{AdaptiveObservationSparseRecovery}}
\newcommand{\hV}{\widehat{V}}
\newcommand{\simu}{\sim_{\textup{unif.}}}

%% file: tex_file/intro.tex
\section{Introduction}
\setcounter{page}{1}

We consider the problem of sparse recovery, a cornerstone problem in learning theory and high-dimensional statistics, with applications to many diverse fields, including medical imaging \cite{lustig2007sparse, graff2015compressive}, computational photography \cite{duarte2008single, gibson2020single} and wireless communication \cite{duarte2011structured, han2013compressive}.
In this problem, we assume there is some underlying ground truth $k$-sparse signal $\vths \in \R^d$, and our goal is to recover it given $n$ (potentially noisy) linear measurements, i.e., from $\vy \defeq \mx \vths + \vxi$ for some measurement matrix $\mx \in \R^{n \times d}$ and some noise vector $\vxi \in \R^n$.
Typically, we are interested in the case where the number of measurements $n$ is much smaller than $d$, and the main statistical measure of merit is how large $n$ has to be to achieve good estimation error for $\vths$.

In this paper, we investigate the question of learning $\vths$ to $\ell_\infty$ error, a task which is closely related to the well-studied question of \emph{variable selection} for sparse linear models \cite{tibshirani1996regression, fan2001variable, candes2007dantzig, meinshausen2010stability, barber2015controlling}.
In many real-world applications of sparse recovery, a primary goal is to select which features of the regression model have significant explanatory power \cite{yamashita2008sparse, belloni2014high, chowdhury2020variable, akyapi2023machine}.
In other words, the task is to find the support of the large elements of the unknown $\vths$. 

This problem is of particular import in overparameterized, high-dimensional settings where $d \gg |\supp(\vths)|$.
By a thresholding argument, observe that this task is more or less equivalent to learning $\vths$ to good $\ell_\infty$ error.
Indeed, recovery in $\ell_\infty$ immediately implies that we can also learn the support of the heavy elements of $\vths$, and conversely, if we can identify this support efficiently, it is (in many natural settings) straightforward to recover $\vths$, since by focusing on those coordinates, we can reduce the problem to standard (i.e., dense) linear regression as long as $n \gtrsim |\supp(\vths)|$.

In the most commonly-studied setting where $\mx$ is an entrywise Gaussian measurement matrix, and the goal is to learn $\vths$ to good $\ell_2$-norm error, the statistical complexity of sparse recovery is by now fairly well-understood.
The seminal work of~\cite{candes2005,candes2006} demonstrated that in the noiseless setting, i.e., $\vxi = \vzero_d$, exact recovery is possible when $n \approx k \log \frac d k$, and moreover, this is achievable with an efficient algorithm ($\ell_1$ minimization).
This sample complexity is tight up to a logarithmic factor, simply by a rank argument.
Follow-up work of~\cite{candes2006robust} demonstrated that for general noise vectors $\vxi \in \R^n$, there is an efficiently-computable estimator $\vth$ which achieves $\ell_2$-norm error
\[
\norm{\vths - \vth}_2 = O\Par{\norm{\vxi}_2},
\]
with the same asymptotic number of measurements, and this recovery rate is optimal.

However despite the large literature on sparse recovery, the sample complexity landscape is significantly less well-understood for recovery in the $\ell_\infty$ norm, and for variable selection in general.
While a number of papers \cite{lounici_2008, ye2010rate, cai2011OMP,huang2017,li2019sub, wainwright2019high} demonstrate upper bounds for this problem, including several that prove error rates for popular algorithms such as LASSO \cite{lounici_2008, ye2010rate, wainwright2019high}, very few lower bounds are known (see Section~\ref{ssec:related} for a more detailed discussion), and moreover, several of these results require additional assumptions on $\vths$ and/or the noise.
For instance, we show in Remark~\ref{remark:compare_error} and Appendix~\ref{app:discuss_gaussian_noise} that the error guarantee in \cite{wainwright2019high} is suboptimal whenever $n = o(k^2)$, even in the mildest noise model we consider (Model~\ref{model:for_each}). To the best of our knowledge, no prior polynomial-time algorithm, including LASSO, was known to achieve the optimal \(\ell_\infty\) error using
\(O(k\log \frac d k)\) samples. This is in stark contrast to the landscape for learning in $\ell_2$, where one can obtain a ``for all'' guarantee for learning any $k$-sparse vector $\vths$ with the same $\mx$.
Additionally, there are very limited lower bounds for learning in $\ell_\infty$ error, and they do not typically match the existing upper bounds.
This state of affairs begs the natural question:
\begin{center}
    {\it Can we characterize the statistical landscape of learning sparse linear models in $\ell_\infty$ error?}
\end{center}
Relatedly, can we understand the sample complexity of variable selection for sparse linear regression?

In this work, we make significant progress on understanding these fundamental
questions. Our main contributions are new sample complexity upper and lower
bounds for variable selection and \(\ell_\infty\) sparse recovery, under various
natural generative models. Before we go into detail about our results, we wish
to emphasize three main algorithmic and conceptual contributions of our investigation.

\paragraph{Optimal oblivious recovery in nearly-linear time.}
We give, to the best of our knowledge, the first polynomial-time algorithm that achieves the optimal \(\ell_\infty\) error in the oblivious Gaussian-noise model with \(n=O(k\log \frac d k)\) measurements. In fact, our algorithm runs in nearly-linear time. This improves over prior LASSO-based guarantees: as noted above, the bound implied by \cite{wainwright2019high} is strictly worse than the optimal Gaussian-noise scale whenever \(n=o(k^2)\).
\paragraph{Adaptivity matters for $\ell_\infty$ sparse recovery.} 
As mentioned, prior works for variable selection and $\ell_\infty$ sparse recovery often required additional assumptions on how the support of the unknown $k$-sparse vector $\vths$ is chosen.
We show that this is inherent: if $\vths$ and $\xi$ are chosen independently of the measurement matrix $\mx$ (the ``oblivious'' or ``for each'' model), then recovery is possible with $n = O(k \log \frac d k)$ measurements in nearly-linear time, but they can be chosen with knowledge of $\mx$ (the ``adaptive'' or ``for all'' model),\footnote{The literature sometimes uses the term ``for all'' model to describe settings with a dependent signal $\vths$ and independent noise $\vxi$, i.e., our ``partially-adaptive'' Model~\ref{model:for_all}. For disambiguation, we primarily refer to our ``for each'' and ``for all'' models as the \emph{oblivious} and \emph{adaptive} models throughout.} then $n = \Omega (k^2)$ measurements are both necessary and (up to a $\log d$ factor) sufficient. In other words, unlike for recovery in $\ell_2$, adaptivity in the choice of the unknown parameters $(\vths, \vxi)$ \emph{provably makes the problem statistically harder}.

\paragraph{A new, canonical choice of error metric.} 
For sparse recovery in the $\ell_2$ norm, it is well-known that the best achievable recovery error is $\norm{\vxi}_2$ (up to constant factors).
However, no such characterization was previously known for $\ell_\infty$.
Various error metrics have been proposed by prior work, discussed thoroughly in Appendix~\ref{append:erro_metric}; however, none were known to yield a tight rate for the achievable error.
In this work, we show strong evidence that the correct error metric in $\ell_\infty$ is
\begin{equation}\label{eq:error_metric}
\mbox{err} (\mx, \vxi) \coloneq \norm{\mx^\top \vxi}_\infty.
\end{equation}
We give the following justifications for this error metric. First, in the oblivious model, we show that this quantity is equivalent to several others considered in the literature, up to appropriate scaling (Lemma~\ref{lem:equiv_err}).
Second, in the adaptive model, we demonstrate nearly-matching upper and lower bounds under the metric \eqref{eq:error_metric}, and we demonstrate that other quantities considered in the literature are provably impossible to achieve in the adaptive model (Lemma~\ref{lem:linf_impossible_adv_1}).

\subsection{Problem statements}\label{ssec:problems}

We now formally define the problems we study in this paper.
Throughout this introduction, we primarily consider the standard setting where $\mx$ has i.i.d.\ entries $\sim \calN(0, \frac 1 n)$; this scaling is convenient as it ensures that $\E [\mx^\top \mx] = \mi_d$.
With this, we now state the variable selection problem we study.

\begin{problem}[Variable selection]\label{prob:variable_selection}
Let $(n, d) \in \N^2$ and $k \in [d]$. Let $\mx \in \R^{n \times d}$ be a known measurement matrix from Model~\ref{model:subg},\footnote{We state a general model of sub-Gaussian matrix ensembles to capture the full generality of our results; the reader may primarily consider the i.i.d.\ entrywise $\Nor(0, \frac 1 n)$ model of $\mx$ for simplicity, which falls under Model~\ref{model:subg}.} and let $(\vths, \vxi) \in \R^d \times \R^n$ be unknown, so that $\nnz(\vths) \le k$ and 
\begin{equation}\label{eq:snr_linfty}\min_{i \in \supp(\vths)} \Abs{\vths_i} > C\norms{\mx^\top \vxi}_\infty,\end{equation}
for a universal constant $C > 0$, represent a signal and noise vector. We observe $(\mx, \vy)$, where $\vy \defeq \mx \vths + \vxi$. Our goal is to output $\supp(\vths) \subseteq [d]$.
\end{problem}

Problem~\ref{prob:variable_selection} is a support recovery problem, that asks to select the relevant variables of the signal vector $\vths$ from the observations $(\mx, \vy)$, under an appropriate signal-to-noise ratio condition \eqref{eq:snr_linfty}. 
As justified previously (and in more detail in Appendix~\ref{app:discuss_gaussian_noise}), we believe that the choice of $\norm{\mx^\top \vxi}_\infty$ on the right-hand side of~\eqref{eq:snr_linfty} is the correct parameterization for this problem.

Next, we formally define the problem of sparse recovery with $\ell_{\infty}$ error.

\begin{problem}[$\ell_\infty$ sparse recovery]\label{prob:linfty_sparse}
Let $(n, d) \in \N^2$ and $k \in [d]$. Let $\mx \in \R^{n \times d}$ be a known measurement matrix from Model~\ref{model:subg}, and let $(\vths, \vxi) \in \R^d \times \R^n$ be unknown so that $\nnz(\vths) \le k$. We observe $(\mx, \vy)$ where $\vy \defeq \mx \vths + \vxi$. Our goal is to output $\vth \in \R^d$ satisfying, for a universal constant $C > 0$,
\begin{equation}\label{eq:error_linfty}\norm{\vth - \vths}_\infty \le C\norms{\mx^\top \vxi}_\infty,\quad \nnz(\vth) \le k.\end{equation}
\end{problem}
To solve an instance of Problem~\ref{prob:variable_selection} with constant $C$, it suffices to solve Problem~\ref{prob:linfty_sparse} with constant $\frac C 2$, and then threshold the coordinates of $\vth$ at $\frac C 2 \norms{\mx^\top \vxi}_\infty$ to recover $\supp(\vths)$. Thus, Problem~\ref{prob:linfty_sparse} is a more general problem (up to the choice of $C$), and is largely the rest of the paper's focus. 

We next define various modeling assumptions used when studying Problem~\ref{prob:linfty_sparse}.
As mentioned, a key conceptual contribution of our paper is that the modeling assumptions play an important role in characterizing the statistical complexity of $\ell_\infty$ sparse recovery.

\begin{model}[Oblivious model]\label{model:for_each}
In the \emph{oblivious model}, $(\vths, \vxi)$ are chosen independently of $\mx$.
\end{model}

\begin{model}[Adaptive model]\label{model:adv_noise}
In the \emph{adaptive model}, we make no independence assumptions on the triple $(\vths, \vxi, \mx)$.
\end{model}

In Model~\ref{model:for_each}, an equivalent viewpoint is that $(\vths, \vxi)$ are first fixed (possibly as samples from a distribution), and then $\mx$ is independently sampled. In Model~\ref{model:adv_noise}, the order is intuitively reversed: first $\mx$ is sampled, and then $(\vths, \vxi)$ can be arbitrarily defined depending on its outcome. 

An algorithm succeeding under Model~\ref{model:adv_noise} is powerful: it can be used in arbitrary adaptively-defined instances of Problem~\ref{prob:linfty_sparse}, with the same $\mx$. This is particularly useful when Problem~\ref{prob:linfty_sparse} is used as a subroutine, e.g., in hyperparameter search or a wrapper optimization algorithm involving $\mx$.

Understanding Problem~\ref{prob:linfty_sparse} under Models~\ref{model:for_each} and~\ref{model:adv_noise} is our main focus, but there are various more fine-grained independence assumptions one could impose. For example, Section~\ref{sec:for_all} investigates a \emph{partially-adaptive model}, where only the noise is viewed as benign.

\subsection{Our results}\label{ssec:results}

\paragraph{Oblivious sparse recovery.} Our first main result is a new algorithm for $\ell_\infty$ sparse recovery in the oblivious model (Model~\ref{model:for_each}). We show that it is possible to solve Problem~\ref{prob:linfty_sparse} in this setting, with a number of samples matching that required for optimal $\ell_2$ sparse recovery.

\begin{theorem}[informal, see Theorem~\ref{thm:l_inf_norm_bound_model1}]
\label{thm:for-each-informal}
    Let $n = \Omega(k \log d)$, and let $\mx \in \R^{n \times d}$ have i.i.d.\ $\Nor(0, \frac 1 n)$ entries.
    There is an estimator which solves Problem~\ref{prob:linfty_sparse} under Model~\ref{model:for_each} with high probability.
    Moreover, the estimator can be computed in nearly-linear time.
\end{theorem}

We prove Theorem~\ref{thm:l_inf_norm_bound_model1} through a simple three-stage method (Algorithm~\ref{alg:IHT+OLS}). Our algorithm first uses iterative hard thresholding (IHT) \cite{blumensath2009iterative}, an $\ell_2$ sparse recovery algorithm, to obtain a warm start. It then estimates the support via thresholding, and solves ordinary least squares on the learned support.
For general sub-Gaussian measurements, the error of Theorem~\ref{thm:for-each-informal} can exceed $\norms{\mx^\top \vxi}_\infty$, but by at most a logarithmic factor; we give a detailed discussion in Sections~\ref{ssec:phase3} and~\ref{ssec:expand_r}.

To our knowledge, this is the first such polynomial-time estimator which achieves these guarantees for $\ell_\infty$ sparse recovery using $\approx k \log d$ samples.\footnote{Nearly-linear time methods have also been developed based on decoding expander graph-based measurements, but typically require explicit design of the measurement matrix. Moreover, these results often operate in the noiseless setting, or only give $\ell_\infty$ recovery under a minimum signal strength assumption; see Section~\ref{ssec:related} for a discussion.}
In particular, Theorem~\ref{thm:for-each-informal} improves upon the current state-of-the-art (LASSO-based) estimator for Problem~\ref{prob:linfty_sparse} under Model~\ref{model:for_each} from \cite{wainwright2019high}, which requires $\Omega(k^2)$ measurements to achieve a comparable error rate (cf.\ Remark~\ref{remark:compare_error}). 

Finally, Theorem~\ref{thm:l_inf_norm_bound_model1} implies an improved runtime for a recent state-of-the-art algorithm for Bayesian sparse linear regression by \cite{KumarSTZ25}. We discuss this application in more detail in Appendix~\ref{app:sample}.

\paragraph{Adaptive sparse recovery.} Our second main result is a new set of nearly-matching upper and lower bounds for the adaptive model (Model~\ref{model:adv_noise}) of variable selection and $\ell_\infty$ sparse recovery.

\begin{theorem}[informal, see Theorems~\ref{thm:linf_iht},~\ref{thm:hard_inst_lower_bound}, and~\ref{thm:support_lb}]
\label{thm:for-all-informal}
Let $n = \Omega(k \log d)$, and let $\mx \in \R^{n \times d}$ have i.i.d.\ $\Nor(0, \frac 1 n)$ entries.
    There is an estimator which solves Problem~\ref{prob:linfty_sparse} under Model~\ref{model:adv_noise} with high probability when $n = \Omega (k^2 \log \frac d k)$, computable in nearly-linear time.
    Moreover, there is no algorithm which can solve either Problem~\ref{prob:variable_selection} or Problem~\ref{prob:linfty_sparse} under Model~\ref{model:adv_noise} with probability $> \half$ if $n = o(k^2)$.
\end{theorem}
Up to a logarithmic factor of the dimension, Theorem~\ref{thm:for-all-informal} settles both the computational and statistical complexity of adaptive $\ell_\infty$ sparse recovery.
Notably, the sample complexity of both parts of Theorem~\ref{thm:for-all-informal} scales quadratically with $k$, as opposed to the linear-in-$k$ scaling in Theorem~\ref{thm:for-each-informal}, as well as in the standard $\ell_2$ sparse recovery setting.
To our knowledge, no such separation had been previously demonstrated in any similar setting.
We conjecture that the lower bound in Theorem~\ref{thm:for-all-informal} can be extended to any $n = o(k^2 \log \frac d k)$, that is, that the tight measurement complexity is $n = \Theta (k^2 \log \frac d k)$; we leave this interesting problem open for future work.

The upper bound in Theorem~\ref{thm:for-all-informal} is once again achieved via an application of IHT. Our main conceptual contribution in demonstrating this result is to define a new notion of the well-studied restricted isometry property (RIP) which we call $\ell_\infty$-RIP (Definition~\ref{def:linf_RIP}), and to demonstrate that Gaussian (and sub-Gaussian) matrices satisfy this condition with high probability when $n = \Omega (k^2 \log \frac d k)$.

The more technically interesting part of Theorem~\ref{thm:for-all-informal} is the lower bound.
At a high level, we demonstrate that the \emph{inverse} of the Gram matrix $\mx^\top \mx$ restricted to any sufficiently small submatrix has a large $\ell_\infty$ operator norm, unless $n = \Omega(k^2)$ (Lemma~\ref{lem:inf_inf_lower_bound}).
We then exhibit a sparse noise vector which has a very large $\ell_\infty$ norm, but which is effectively killed off by the measurements $\mx$ as long as $n = o(k^2)$, and hence cannot be detected unless we have sufficiently many measurements.

\paragraph{Variable selection with partial adaptivity.} A natural question is whether or not one can circumvent the quadratic lower bounds of the adaptive model in an intermediate setting which interpolates between Models~\ref{model:for_each} and~\ref{model:adv_noise}.
Towards understanding this possibility, we demonstrate nontrivial recovery guarantees for variable selection (Problem~\ref{prob:variable_selection}), in a model we call the \emph{partially adaptive} model, where the noise $\vxi$ is independent of $\mx$, but $\vths$ may be adaptive (Model~\ref{model:for_all}).

We demonstrate in Theorem~\ref{thm:partial-adaptivity} that if the learner is allowed to mask the effect of certain coordinates when querying observations (Algorithm~\ref{alg:obs_gen}), then $\approx k\log d\log k$ measurements suffice to solve variable selection under partial adaptivity. Our algorithm iteratively applies thresholding after masking an estimated support, which we show makes geometric progress on the residual. Although our result holds under a nonstandard observation model, we believe it highlights the possibility of going beyond Theorem~\ref{thm:for-all-informal} even under partial adaptivity. Indeed, a key structural result that we leverage in our algorithm is that threshold-based support learning has few false positives under Model~\ref{model:for_all}; we believe this observation will prove useful in future investigations of the partially-adaptive setting.

\subsection{Related work}\label{ssec:related}

\textbf{$\ell_1$-convex relaxation.}
Since \cite{natarajan1995} showed that $\ell_0$-minimization for linear regression is NP-hard in general, extensive efforts
\cite{tibshirani1996regression, chen2001atomic, candes2005, candes2006, candes2007dantzig}
have focused on $\ell_1$ minimization as a tractable convex surrogate. Under RIP-type assumptions and with measurement complexity $n = \Theta(k \log \frac d k)$, both the Lasso \cite{wainwright2019high} and the Dantzig selector \cite{candes2007dantzig} are known to achieve optimal $\ell_2$ estimation error rates. Beyond $\ell_2$ recovery, \cite{zhao2006model} showed that the Lasso achieves exact support recovery when the nonzero coefficients are sufficiently strong and the measurements satisfy an irrepresentablity condition, requiring $n = \Omega(k \log d)$ in the oblivious model and $n = \Omega(k^2 \log d)$ in the adaptive model.
For $\ell_\infty$ guarantees, \cite{lounici_2008} proved that both the Lasso and the Dantzig selector attain $\approx \sigma$ error under a pairwise incoherence condition, which requires $n = \Omega(k^2)$ (Lemma~\ref{lem:LB_inf_RIP}).
Similarly, \cite{ye2010rate} derived comparable $\ell_\infty$ bounds under an $\ell_\infty$-curvature condition, again with $n = \Omega(k^2 \log d)$. More recently, \cite{wainwright2019high} showed that under mutual incoherence, the Lasso satisfies a refined $\ell_\infty$ error bound which is achieved with $n = \Theta(k \log d)$ in the oblivious model and $n = \Theta(k^2 \log d)$ in the adaptive model for sub-Gaussian designs.
As we show in Remark~\ref{remark:compare_error}, this bound matches the minimax-optimal $\approx \sigma$ $\ell_\infty$ error only when $n = \Omega(k^2)$. 

\textbf{Greedy selection under $\ell_0$ constraints.}
Beginning with seminal works of \cite{mallat1993matching,pati1993orthogonal}, a substantial line of research
\cite{needell2008greedy,needell2008cosamp,cai2011OMP,Jain2011,donoho2012StOMP}
has developed greedy heuristics for approximating the intractable $\ell_0$ minimization problem.
\cite{pati1993orthogonal} introduced Orthogonal Matching Pursuit (OMP), which iteratively selects the column most correlated with the current residual and removes its contribution via orthogonal projection. 
OMP \cite{zhang2011} and its variants, including ROMP \cite{needell2008greedy}, CoSaMP \cite{needell2008cosamp}, and OMPR \cite{Jain2011}, admit similar RIP-based analyses and achieve $\ell_2$ estimation error $\approx \|\vxi\|_2$ with measurement complexity $n=\Omega(k\log^{O(1)} d)$.
These guarantees are comparable, or sometimes slightly weaker, than the sharp bounds obtained for IHT \cite{price2021}.
On the other hand, \cite{cai2011OMP} and \cite{huang2017} establish $\ell_\infty$ recovery guarantees for OMP and its variant SDAR under incoherence-type assumptions and a Gaussian noise model, which lead to a higher measurement requirement of $n=\Omega(k^2\log d)$. These greedy methods typically incur higher computational costs than IHT and our Algorithm~\ref{alg:oblivious-general}, as they require solving a least-squares problem over a size-$k$ support at each iteration, whereas IHT performs only simple gradient descent and thresholding operations and Algorithm~\ref{alg:oblivious-general} only solves OLS once on top of calling IHT.

\textbf{Expander-based methods.}
A parallel line of work
\cite{xu2007efficient,jafarpour2009efficient,indyk2008near,berinde2008practical}
uses sparse binary measurement matrices, typically the adjacency matrices of bipartite expanders, to enable sparse recovery, primarily in noiseless settings.
For example, \cite{xu2007efficient} showed that exact recovery with a for each guarantee (Model~\ref{model:for_each}) is possible using $n=O(k\log d)$ measurements and decoding time $T=O(d\log d)$.
Subsequent work improved efficiency: \cite{sarvotham2006sudocodes} reduced decoding time to $T=O(k\log d\log k)$, while \cite{wu2012optimal} sharpened the measurement complexity to $n=O(k)$ under additional signal model assumptions. Notably, such results require strong control of the measurement matrix.

Despite these favorable guarantees in the noiseless regime, expander-based decoders are generally less robust to noise, and analyses in noisy settings remain limited
\cite{jafarpour2009efficient,acharya2017improved,li2019sub}.
Notably, \cite{jafarpour2009efficient} established robustness for approximately $k$-sparse signals, while \cite{acharya2017improved} showed that $n=\Theta(k^2\log d)$ measurements are both necessary and sufficient for support recovery in the adaptive model under binary measurements.
Furthermore, \cite{li2019sub} provided $\ell_\infty$ guarantees in the oblivious model, but required all signal coordinates to exceed the noise level by a significant margin.
Iterative expander-based algorithms such as EMP \cite{indyk2008near} and SSMP \cite{berinde2008practical} achieve optimal $\ell_2$ recovery in adaptive models, but their computational cost is comparable to greedy methods like IHT and OMP, and substantially higher than that of combinatorial decoders.

\textbf{Lower bounds on support recovery}. Lower bounds for support recovery remain comparatively less explored. Classic works such as \cite{candes2013estimate, scarlett2019introductory, wainwright2009information} establish information-theoretic lower bounds on the $\ell_2$ estimation error via Fano’s inequality, implying that $n \ge k \log \frac d k$ measurements are necessary under their respective scalings. In contrast, under our normalization this line of analysis yields a unified risk lower bound without an explicit measurement constraint (see Appendix~\ref{app:discuss_gaussian_noise}). For support recovery, \cite{wainwright2006sharp} shows that the Lasso consistently identifies the true support under Gaussian designs only if $n = \Omega\left(k \log \frac d k\right)$. This result was extended by \cite{fletcher2009necessary} to maximum-likelihood estimators, yielding a necessary condition $n = \Omega(\tfrac{k \log d/k}{\mathrm{SNR}\cdot\mathrm{MAR}})$, where the denominator quantifies the signal strength. More recently, \cite{gamarnik2017high} identified an additional “all-or-nothing” threshold for support recovery of binary signals, occurring at $n = k \log d \cdot \log \frac{k}{1+\sigma^2}$.

%% file: tex_file/preliminaries.tex
\section{Preliminaries}
\label{sec:prelim}

In this section we develop some preliminaries for the rest of the paper. 
\subsection{Notation}\label{ssec:notation}

\textbf{General notation.} We denote matrices in capital boldface and vectors in lowercase boldface. We define $[n] \defeq \{i \in \N : i \le n\}$. When $S$ is a subset of a larger set clear from context, $S^c$ denotes its complement. For random variables $X$ and $Y$, $X\perp Y$ denotes that $(X, Y)$ are independent. For an event $\calE$, we use $\ind_{\calE}$ to denote its associated $0$-$1$ indicator variable.

To simplify expressions, we henceforth assume $k$ is at least a sufficiently-large universal constant in future proofs. 
We also always let $\{\vx_i \in \R^n\}_{i \in [d]}$ denote the columns of the measurement matrix $\mx$ when clear from context.

\textbf{Vectors.} For a vector $\vv \in \bbR^d$, we let $\supp(\vv) = \cbra{i \in [d]: \vv_i \neq 0}$ and let $\ktwonorm{\vv}$ be the $\ell_2$-norm of its top-$k$ largest elements in the absolute value. We let $\vzero_d$ and $\vone_d$ denote the all-zeroes and all-ones vectors in $\R^d$. We also define $\ve_i$ as the $i^{\text{th}}$ standard basis vector.

For $\vv \in \R^d$ and $k \in [d]$, we let $H_k(\vv) \in \R^d$ keep the $k$ largest coordinates of $\vv$ by magnitude (breaking ties in lexicographical order), i.e., the ``head,'' and set all other coordinates to $0$.

\textbf{Matrices.} We let $\id_d$ denote the $d \times d$ identity matrix, and $\id_S$ is the identity on $\R^S$ for an index set $S$. For $p \ge 1$ (including $p = \infty$), applied to a vector argument, $\norm{\cdot}_p$ denotes the $\ell_p$ norm. For $p, q \ge 1$ and a matrix $\mm \in \R^{n \times d}$, we also use the notation
$\norm{\mm}_{p \to q} \defeq \max_{\substack{\vv \in \R^d \mid \norm{\vv}_p \le 1}} \norm{\mm \vv}_q$. We use $\normf{\cdot}$ and $\normop{\cdot}$ to denote the Frobenius and ($2 \to 2$) operator norms of a matrix argument.

A helpful observation used throughout is that $\norm{\cdot}_{\infty \to \infty}$ is the largest $\ell_1$ norm of a row.

\textbf{Indexing.} For any $i \in [d]$, $\vv \in \R^d$, $\mx\in\R^{n\times d}$, we let $\vv_{-i} \in \R^{d-1}$ drop the $i^{\text{th}}$ element of $\vv$ and $\mx_{:-i}\in \R^{n \times (d-1)}$ drop the $i^{\text{th}}$ column of $\mx$. For a vector $\vv \in \R^d$ and $S \subseteq [d]$, we use $\vv_S \in \R^S$ to denote its restriction to its coordinates in $S$. We let $\nnz(\cdot)$ denote the number of nonzero entries in a matrix or vector argument.  For $\mm \in \R^{n \times d}$, we use $\mm_{i:}$ to denote its $i^{\text{th}}$ row for $i \in [n]$, and $\mm_{:j}$ to denote its $j^{\text{th}}$ column for $j \in [d]$. When $S \subseteq [n]$ and $T \subseteq [d]$ are row and column indices, we let $\mm_{S \times T}$ be the submatrix indexed by $S$ and $T$; if $T = [d]$ we simply use $\mm_{S:}$ and similarly, we define $\mm_{:T}$. We fix the convention that transposition is done prior to indexing, i.e., $\mm_{S\times T}^\top \defeq [\mm^\top]_{S\times T}$.

\textbf{Sub-Gaussian distributions.}
We give a simplified introduction to sub-Gaussian distributions. For a detailed discussion, we refer to Chapter 2 and 3 of \cite{vershynin2018high}. 

\begin{definition}[sub-Gaussian distribution]\label{def:subg}
We say that a random variable $X \in \R$ is $\sigma^2$-sub-Gaussian with mean $\mu$, which we denote by $X \sim \subg(\mu, \sigma^2)$, if $\E[X] = \mu$, and
\[
\E\Brack{\exp\Par{\lam \Par{X - \mu}}} \le \exp\Par{\frac{\lam^2 a^2}{2}},\text{ for all } \lam \in \R.
\]
\end{definition}

The following facts will be very helpful in manipulating sub-Gaussian random variables. The first follows simply by applying Definition~\ref{def:subg} and applying independence appropriately.

\begin{fact}\label{fact:composition}
If $X \sim \subg(\mu, \sig^2)$ and $a \in \R$, then $aX \sim \subg(a\mu, a^2\sig^2)$, and if $Y \sim \subg(\nu, \tau^2)$ where $X \perp Y$, then $X + Y \sim \subg(\mu + \nu, \sig^2 + \tau^2)$.
\end{fact}

\begin{lemma}[Hoeffding's inequality, Theorem 2.2.1, \cite{vershynin2018high}]\label{lem:hoeffding}
If $X \sim \subg(\mu, \sig^2)$, then for all $t \ge 0$,
\[\Pr\Brack{\Abs{X - \mu} \ge t} \le 2\exp\Par{-\frac{t^2}{2\sig^2}}.\]
\end{lemma}

\begin{lemma}[Proposition 2.6.1, \cite{vershynin2018high}]
\label{lem:subg_property}
Each of these statements implies each of the others, for some constants $C_1, C_2 > 0$ (where the constants may change in the different directions of implication).
\begin{enumerate}
    \item $X \sim \subg(\mu, \sigma^2)$.
    \item $\E[X] = \mu$, and for any $p\in \N$, $\expect{|X - \mu|^p} \le (C_1\sigma\sqrt{p})^p$.\label{item:moments}
    \item $\E[X]  = \mu$, and $\E[\exp(\frac{(X-\mu)^2}{C_2 \sigma^2})] \le C_2$.\footnote{In \cite{vershynin2018high}, Item~\ref{item:mgfsquare} is listed with a bound of $2$ on the right-hand side. However, examining the proof shows any constant bound suffices for concluding sub-Gaussianity (through an appropriate tail bound).}\label{item:mgfsquare}
\end{enumerate}
\end{lemma}

\begin{lemma}\label{lem:anticonc}
    If $X \sim \subg(0, C\sigma^2)$ has variance $\sigma^2$ for some constant $C > 0$, then there exist constants $c_1, c_2 > 0$ such that $|X| \sim \subg(\mu, c_2\sig^2)$ for $\mu \ge c_1 \sig$.
\end{lemma}
\begin{proof}
    By Item~\ref{item:moments} in Lemma~\ref{lem:subg_property}, we know that $\expect{X^4} = O(\sigma^4)$. Applying the Paley-Zygmund inequality to $Y = X^2$, there exists a constant $c > 0$ such that 
    \begin{equation*}
        \prob{X^2 \ge \tfrac12\expect{X^2}} \ge \frac{\expect{X^2}^2}{4\expect{X^4}} \ge c.
    \end{equation*}
    Hence, $\expect{|X|} \ge c_1\sigma$ for some $c_1 > 0$, proving the claim about the mean. Next, by Item~\ref{item:mgfsquare} in Lemma~\ref{lem:subg_property}, we know that there exists some $C_2 > 0$ such that 
    \begin{align*}
    \E\Brack{\exp\Par{\frac{\Par{|X| - \E|X|}^2}{2C_2\sig^2}}} \le  \E\Brack{\exp\Par{\frac{2X^2 + 2(\E|X|)^2}{2C_2\sig^2}}} \le C_2\exp\Par{\frac{1}{C_2}}.
    \end{align*}
    Thus, there is a constant such that Item~\ref{item:mgfsquare} holds, so the sub-Gaussian parameter is $O(\sig^2)$.
\end{proof}

\subsection{Sparse recovery preliminaries}\label{ssec:sparse_recovery}

In this section, we recall some standard results from the sparse recovery literature.

\textbf{Regularity assumptions for sparse recovery.} We introduce two structural properties of $\mx$ that are commonly used to make sparse recovery tractable. 
For a more detailed introduction to these properties, we refer the reader to Chapter 7 of~\cite{wainwright2019high}.

\begin{definition}[Restricted isometry property]
    \label{def:RIP}
    Let $(\eps, s) \in (0, 1) \times [d]$.
    We say $\mx \in \bbR^{n\times d}$ satisfies the $(\epsilon, s)$-restricted isometry property, or $\mx$ is $(\epsilon, s)$-$\RIP$, if for all $\vth \in \bbR^d$ with $\nnz(\vth) \leq s$, 
    \begin{equation*}
        (1 - \epsilon) \twonorm{\vth}^2 \leq \twonorm{\mx\vth}^2 \leq (1 + \epsilon) \twonorm{\vth}^2.
    \end{equation*}
    An equivalent condition is that $\vlam([\mx^\top \mx]_{S \times S}) \in [1 - \eps, 1 + \eps]^s$ for all $S \subseteq [d]$ with $|S| \le s$.
\end{definition}
Intuitively, RIP implies $\mx$ acts as an
approximate isometry on sparse vectors, so $[\mx^\top \mx]_{S\times S}$ is well-conditioned for any sparse support $S$. It is well-known that Definition~\ref{def:RIP} is satisfied by various random matrix ensembles. In this paper, we primarily focus on $\mx$ drawn from sub-Gaussian ensembles, which we formally define here for ease of reference.

\begin{model}\label{model:subg}
Let $\mx \in \R^{n\times d}$ have i.i.d.\ entries $\sim \subg(0, \tfrac C n)$, with variance $\tfrac 1 n$, for a constant $C > 0$.
\end{model}

Next, we state a useful property of sub-Gaussian matrix ensembles.

\begin{proposition}[Theorem 9.2, \cite{foucart13}]
    \label{prop:RIP_sub_gaussian}
Let $\delta, \eps \in (0, \half)^2$. Under Model~\ref{model:subg}, for any $s \in [d]$, $\mx$ is $(\eps, s)$-RIP with probability $\ge 1 - \delta$ if, for an appropriate constant,
    \begin{equation*}
        n = \Omega\Par{\frac{s\log\frac{d}{s} + \log\frac{1}{\delta}}{\epsilon^2}}.
    \end{equation*}
\end{proposition}
For more RIP matrix ensembles, including those based on sampling bounded orthonormal systems and real trignometric polynomials, we refer the reader to Chapter 12 of \cite{foucart13}.

Another common condition for sparse recovery is through the lens of (approximate) orthogonality.

\begin{definition}[Pairwise incoherence]
    \label{def:PI}
    Let $\alpha \in (0, 1)$.
    We say $\mx \in \bbR^{n\times d}$ with columns $\{\vx_i\}_{i \in [d]}$ satisfies $\alpha$-pairwise incoherence, or $\mx$ is $\alpha$-PI, if 
    \begin{equation*}
        \max_{(i, j) \in [d] \times [d]} \Abs{\Brack{\mx^\top\mx - \id}_{ij}} \leq \alpha.
    \end{equation*}
\end{definition}

\begin{proposition}[Lemma 6.26, \cite{wainwright2019high}]\label{prop:pi}
Let $\alpha, \delta \in (0, \half)^2$. Under Model~\ref{model:subg}, $\mx$ is $\alpha$-PI with probability $\ge 1 - \delta$ if, for an appropriate constant,
    \begin{equation*}
        n = \Omega\Par{\frac{\log \frac d \delta}{\alpha^2}}.
    \end{equation*}
\end{proposition}
\textbf{$\ell_2$ sparse recovery in nearly-linear time.} Our algorithms use $\ell_2$ sparse recovery, an extensively studied primitive, as a subroutine. In Algorithm~\ref{alg:IHT}, we recall one famous sparse recovery algorithm, iterative hard thresholding (IHT) \cite{blumensath2009iterative}, whose $\ell_2$ error guarantees are well-understood.

\begin{algorithm}[ht]
\DontPrintSemicolon
    \caption{$\IHT(\mx, \vy, k, R, r)$}
    \label{alg:IHT}
    $\vth^{(0)} \gets \vzero_d$\;
    \If{$r \ge R$}{\Return {$\vth^{(0)}$}}
    $T \gets \lceil\log_2 \frac{R}{r}\rceil$\;
    \For {$t = 0, 1, \ldots, T-1$}{
    $\vth^{(t+1)} \gets H_k\paren{\vth^{(t)} + \mx^\top (\vy - \mx \vth^{(t)})}$\;
    }
    \Return {$\vth \gets \vth^{(T)}$}
\end{algorithm}

\begin{lemma}[Theorem 4.8, \cite{price2021}]
    \label{lem:IHT_l2_recover}
    In Problem~\ref{prob:linfty_sparse}, if $\mx$ is $(0.14, 3k)$-RIP, then for any  $R \ge \twonorm{\vths}$, the output $\vth$ of Algorithm~\ref{alg:IHT} satisfies\footnote{A minor difference is that \cite{price2021} derives a bound in terms of $\twonorm{\vxi}$, whereas we require a bound on $\norm{\mx^\top \vxi}_{2,3k}$. This mismatch can be resolved by refining Lemma 4.7 of \cite{price2021}: instead of applying the coarse bound $\twonorm{\mx^\top_{S:}\vxi} \le \normop{\mx_{S:}^\top}\twonorm{\vxi}$, one can directly control $\twonorm{\mx^\top_{S:}\vxi}$ via the tighter quantity $\norm{\mx^\top \vxi}_{2,3k}$.} 
    \begin{equation*}
        \twonorm{\vth - \vths} \leq r + 5\norm{\mx^\top \vxi}_{2,3k} \le r + 5\sqrt{3k}\infbound,\quad \Abs{\supp(\vth)} \le k.
    \end{equation*}
\end{lemma}

%% file: tex_file/model1.tex
\section{Oblivious \texorpdfstring{$\ell_\infty$}{L-inf} Sparse Recovery}
\label{sec:for_each}
In this section, we develop our estimation algorithm to solve Problem~\ref{prob:linfty_sparse} under the oblivious Model~\ref{model:for_each}. 

In Algorithm~\ref{alg:IHT+OLS}, we decompose the task into three phases.
\begin{enumerate}
    \item In the first phase, we run $\IHT$ (Algorithm~\ref{alg:IHT}) to obtain a warm start and finer control on $\norms{\vths}_2$.
    \item In the second phase, we use a simple thresholding procedure to estimate the support of $\vths$.
    \item In the third phase, we perform an ordinary least squares (OLS) step on the learned support.
\end{enumerate}
The analysis of the first phase black-box applies Lemma~\ref{lem:IHT_l2_recover}. After stating some technical preliminaries  in Section~\ref{ssec:helper}, we analyze the second phase in Section~\ref{ssec:phase2} and the third phase in Section~\ref{ssec:phase3}.

Since Algorithm~\ref{alg:IHT+OLS} proceeds in multiple phases, the estimation in later phases generally depends on the outcomes of the preceding ones. To handle this dependence cleanly, we first establish a basic equivalence principle: bounds proved \emph{for every fixed design} are equivalent to bounds proved \emph{for any independent random design}. We will use these two viewpoints interchangeably henceforth.

\begin{lemma}
    \label{lem:fix_param_indep_model}
    Let $(\Omega_1, \calF_1)$ and $(\Omega_2, \calF_2)$ be two measurable spaces, and let $R_1: \Omega_1\to \calX_1$ and $R_2: \Omega_2 \to \calX_2$ be two random variables. Let $\calE \subset \calX_1 \times \calX_2$ be an measurable event. 
    Then 
    \begin{equation*}
        \bbP_{R_1}\bra{(R_1, R_2) \in \calE} \le \delta ,  \text{ for all } R_2 \in \calX_2\iff \bbP_{R_1, R_2}\bra{(R_1, R_2) \in \calE} \le \delta, \text{ for any } R_2 \perp R_1.
    \end{equation*}
\end{lemma}
\begin{proof}
The $\impliedby$ direction follows by taking $R_2$ to be deterministic (a Dirac measure).
To show the $\implies$ direction, notice that 
    \begin{equation*}
        \bbP_{R_1,R_2}\bra{(R_1, R_2) \in \calE} = \bbE_{\R_2}[\bbP_{R_1}\bra{(R_1, R_2)\in \calE \mid R_2}] = \bbE_{R_2}\bra{\bbP_{R_1}\bra{(R_1, R_2)\in \calA}},
    \end{equation*}
    where the last equality is by the fact that $R_1 \perp R_2$.
    By the assumption, we obtain 
    \begin{equation*}
        \bbP_{R_1, R_2}\bra{(R_1, R_2)\in \calE} = \bbE_{R_2}\bra{\bbP_{R_1}\bra{(R_1, R_2) \in \calE}} \le \bbE_{R_2}\bra{\delta} = \delta.
    \end{equation*}
\end{proof}
\begin{algorithm}[ht]
\DontPrintSemicolon
    \caption{$\OSR(\mx, \vy, k, R, r, c)$}
    \label{alg:IHT+OLS}
    $(\mx, \vy) \gets (3\mx, 3\vy)$\;
    Evenly divide observations $(\mx, \vy)$ into $\mx = \begin{bmatrix}
        \mx^{(1)}\\
        \mx^{(2)}\\
        \mx^{(3)}\\
    \end{bmatrix}, \vy = \begin{bmatrix}
        \vy^{(1)}\\
        \vy^{(2)}\\
        \vy^{(3)}
    \end{bmatrix}$\;
    $\vhth \gets \IHT(\mx^{(1)}, \vy^{(1)}, k, R, \sqrt k r)$ \label{line:iht}\;
    $(\vr^{(2)}, \vr^{(3)}) \gets (\vy^{(2)} - \mx^{(2)} \vhth, \vy^{(3)} - \mx^{(3)} \vhth)$ \label{line:residual}\;
    $L \gets \{i \in [d] : |\mx^{(2)\top}_{i:} \vr^{(2)}| \geq \frac r c\}$\label{line:support_id}\;
    $\vth \gets \vhth + [\mx^{(3)\top}\mx^{(3)}]^{-1}_{L \times L} \mx^{(3)\top}_{L:}\vr^{(3)}$\label{line:ols}\;
    \Return {$\vth$}
\end{algorithm}
\subsection{Helper technical lemmas}\label{ssec:helper}

Before we dive into the analysis of Algorithm~\ref{alg:IHT+OLS}, we first introduce several helper lemmas about sub-Gaussian ensembles (Model~\ref{model:subg}), which will be used later. 

\begin{lemma}
    \label{lem:cross_term_fix}
Let $\delta \in (0, \half)$ and let $\vv \in \R^d$, $\vu \in \R^n$ be fixed. Under Model~\ref{model:subg}, with probability $\ge 1 - \delta$,
\[\Abs{\inprod{\vu}{\mx\vv}} \le 2\norm{\vu}_2\norm{\vv}_2\sqrt{\frac{C\log\frac 1 \delta}{n}}.\]
\end{lemma}
\begin{proof}
By Fact~\ref{fact:composition} entrywise, $\inprod{\vu}{\mx\vv} \sim \subg(0, \frac C n \norm{\vv}_2^2\norm{\vu}_2^2)$. The claim is now Lemma~\ref{lem:hoeffding}.
\end{proof}

We conclude the following corollary holds using Lemma~\ref{lem:fix_param_indep_model}.

\begin{corollary}
    \label{cor:cross_term_indep}
    Let $\delta \in (0, \half)$,
    and under Model~\ref{model:subg}, let $\vv \in \R^d$, $\vu \in \R^n$ be random variables satisfying $(\vu, \vv) \perp \mx$ and $\norm{\vv}_2 \le \alpha$, $\norm{\vu}_2 \le \beta$ with probability $1$. Then with probability $\ge 1 - \delta$,
    \begin{equation*}
        \innp{\vu, \mx\vv} \le 2\alpha\beta\sqrt{\frac{C\log \frac 1 \delta}{n}}.
    \end{equation*}
\end{corollary}

Next, we show an important technical observation, that intuitively translates to the $\ell_\infty$-operator norm bound of $\mx^\top \mx$ restricted to small submatrices not holding for ``typical'' vectors. For a $k \times k$ submatrix of $\mx^\top \mx$ (or its inverse), Proposition~\ref{prop:pi} suggests that we need $n \approx k^2$ for the $\ell_\infty$-operator norm to be bounded by $O(1)$. The following result shows that if we parameterize the operator norm differently, $n \approx k$ samples suffices for a sharper bound on an independent vector to hold.

\begin{lemma}
    \label{lem:noise_linf_bound_fix}
Let $\delta \in (0, \half)$, $n = \Omega(k\log \frac d \delta)$ for an appropriate constant, and $k \in [d]$. Under Model~\ref{model:subg}, for any fixed $S \subseteq [d]$ with $|S| = s \le k$, and fixed $\vv \in \R^n$, with probability $\ge 1 - \delta$ the following hold:
    \begin{align}
        \label{eq:xxinv_v_l2}
        &\infnorm{\xxinv_{S\times S}\mx_{S:}^\top \vv} \le 8\twonorm{\vv}\sqrt{\frac{C\log \frac{d}{\delta}}{n}},\\
        \label{eq:xxinv_v_inf}
        \frac 1 6 \infnorm{\mx_{S:}^\top\vv}\le &\infnorm{\xxinv_{S\times S}\mx_{S:}^\top \vv} \le 6\infnorm{\mx_{S:}^\top\vv}.
    \end{align}
\end{lemma}
\begin{proof}
By Proposition~\ref{prop:RIP_sub_gaussian}, with probability $1-\frac \delta 4$, $\mx$ satisfies $(\half, k)$-RIP, and under this event, applying the RIP definition to standard basis vectors yields that for all $i \in [d]$, we have 
$\twonorm{\vx_i}^2 \in [\half, \frac 3 2]$.

To prove the statement, it suffices to control the quantity, for all $i \in S$,
\[\Abs{\ve_i^\top [\mx^\top \mx]^{-1}_{S\times S}\mx_{S:}^\top \vv}.\]
For any $i\in S$, we let $\mx_{:S-i}\in\bbR^{n\times (s-1)}$ be the submatrix of $\mx_{:S}$ without its $i^{\text{th}}$ column. Now fix an index $i \in S$. Displaying $S \times S$ matrices so that index $i$ is in the top-left corner gives
\begin{align*}
[\mx^\top \mx]_{S\times S}
=
\begin{bmatrix}
a & \vb^\top\\
\vb & \mc
\end{bmatrix},
\text{ where }
a \defeq \vx_i^\top \vx_i,\;
\vb \defeq \mx_{S-i:}^\top \vx_i,\;
\mc \defeq \mx_{S-i:}^\top \mx_{:S-i}.
\end{align*}

By the Schur complement formula,
\begin{equation*}
[\mx^\top \mx]^{-1}_{S\times S}
=
\frac{1}{a-\vb^\top \mc^{-1}\vb}
\begin{bmatrix}
1 & -\vb^\top \mc^{-1}\\
-\mc^{-1}\vb & \mc^{-1}\vb\vb^\top\mc^{-1}+(a-\vb^\top\mc^{-1}\vb)\mc^{-1}
\end{bmatrix}.
\end{equation*}
We introduce the helpful notation
\begin{equation*}
\matp_i := \id_n - \mx_{:S-i}(\mx_{S-i:}^\top\mx_{:S-i})^{-1}\mx_{S-i:}^\top,
\quad
\vu_i := \mx_{:S-i}(\mx_{S-i:}^\top\mx_{:S-i})^{-1}\mx_{S-i:}^\top \vv,
\end{equation*}
so $\matp_i$ is the orthogonal projector onto the range of $\mx_{:S-i}$.
A direct calculation yields
\begin{equation*}
a-\vb^\top\mc^{-1}\vb = \vx_i^\top \matp_i \vx_i,
\quad
\ve_i^\top [\mx^\top\mx]^{-1}_{S\times S}\mx_{S:}^\top \vv
=
\frac{\vx_i^\top \matp_i \vv}{\vx_i^\top \matp_i \vx_i}
=
\frac{\vx_i^\top \vv-\vx_i^\top \vu_i}{\vx_i^\top \matp_i \vx_i}.
\end{equation*}
Thus, 
\begin{equation}\label{eq:key_ratio}
\infnorm{\xxinv_{S\times S}\mx_{S:}^\top \vv} = \max_{i\in S} \frac{\Abs{\vx_i^\top \matp_i \vv}}{\vx_i^\top \matp_i \vx_i} = \max_{i\in S}\frac{\Abs{\vx_i^\top \vv-\vx_i^\top \vu_i}}{\vx_i^\top \matp_i \vx_i}.
\end{equation}
We proceed to establish concentration of both the numerator and denominator of \eqref{eq:key_ratio}.

\textbf{Controlling the denominator of \eqref{eq:key_ratio}.}
The matrix $\matp_i$ is an orthogonal projection with rank
$r = n+1-s \ge \frac n 2$, and $\vx_i \perp \matp_i$. Hence, for all $i \in [s]$,
\[\E\Brack{\vx_i^\top \matp_i\vx_i} = \inprod{\E\Brack{\vx_i\vx_i^\top}}{\matp_i} = \frac 1 n \Tr(\matp_i) = \frac r n \ge \half.\]
Next, the Hanson-Wright inequality (Theorem 6.2.2, \cite{vershynin2018high}) states that there is a constant $c > 0$ such that for all $t \ge 0$,
\begin{align*}
\Pr\Brack{\vx_i^\top \matp_i \vx_i \le \E[\vx_i^\top \matp_i \vx_i] - t} \le \exp\Par{-c\min\Par{\frac{t^2 n^2}{C^2 r}, \frac{tn}{C}}}.
\end{align*}
Combining the above two displays, and applying a union bound over $i \in S$, we have
\begin{equation}\label{eq:hanson_wright}\Pr\Brack{\exists i \in S \mid \vx_i^\top \matp_i \vx_i \le \frac 1 4} \le s\exp\Par{-c\min\Par{\frac{n}{16C^2},\frac{n}{4C}}} \le \frac \delta 4,\end{equation}
for sufficiently large $n = \Omega(\log \frac d \delta)$ as stated, using $s \le d$.
Also, RIP implies $\vx_i^\top \matp_i \vx_i \le \twonorm{\vx_i}^2 \le \frac 3 2$.

\paragraph{Control of $\vx_i^\top \matp_i \vv$.}
Since $\matp_i$ is a projection operator, $\twonorm{\matp_i\vv} \le \twonorm{\vv}$.
Conditioning on $(\matp_i,\vv)$ and using $\vx_i \perp (\matp_i, \vv)$, we thus have
$\vx_i\matp_i\vv \sim \subg(0, \frac C n \twonorm{\matp_i\vv}^2)$.
Hence,
\begin{equation*}
\prob{\Abs{\vx_i^\top \matp_i\vv} > t} \le 2\exp\paren{-\frac{ n t^2}{2C\|\vv\|_2^2}},
\end{equation*}
for all $i \in S$, by Lemma~\ref{lem:hoeffding}.
Now applying a union bound over $i\in S$ yields
\begin{equation*}
\prob{\exists i \in S : \Abs{\vx_i^\top \matp_i \vv} > 2\|\vv\|_2\sqrt{\frac{C\log\frac{d}{\delta}}{n}}}\le \frac \delta 4, 
\end{equation*}
which further implies~\eqref{eq:xxinv_v_l2} when plugged into \eqref{eq:key_ratio}, as long as the event in \eqref{eq:hanson_wright} fails.

\paragraph{Control of $\vx_1^\top \vu_i$.}
Conditioned on $\mx$ satisfying $(\half,k)$-RIP, we have
\begin{equation*}
\twonorm{\vu_i} 
\le
\norm{\mx_{:S-i}}_{2\to 2} \cdot \normop{[\mx_{S-i:}^\top\mx_{:S-i}]^{-1}} \cdot \twonorm{\mx_{S-i:}^\top \vv}
\le
4\sqrt{s}\infnorm{\mx^\top_{S:}\vv}
\le
4\sqrt{k}\infnorm{\mx^\top_{S:}\vv}.
\end{equation*}

Since $\vu_i$ is independent of $\vx_i$, we have by Fact~\ref{fact:composition} that
$\vx_i^\top \vu_i \sim \subg(0, \frac C n \twonorm{\vu_i}^2)$, and thus
\begin{equation*}
\prob{\Abs{\vx_i^\top \vu} > t}
\le
2\exp\paren{-\frac{nt^2}{2C\twonorm{\vu}^2}}
\le
2\exp\paren{-\frac{n t^2}{32kC\infnorm{\mx^\top_{S:} \vv}^2}}.
\end{equation*}

Taking $t = \half\infnorm{\mx_{S:}^\top \vv}$ and applying a union bound over $i\in S$ yields
\begin{equation*}
\prob{\exists i : \Abs{\vx_i^\top \vu} > \half\infnorm{\mx_{S:}^\top \vv}}\le \frac \delta 4, 
\end{equation*}
for sufficiently large $n = \Omega(k\log \frac d \delta)$. Conditioning on all four random events used so far in the proof gives the failure probability. Now letting $j \in S$ be the index that achieves $|\vx_{j}^\top\vv| = \infnorm{\mx_{S:}^\top \vv}$, and letting $i \in S$ be the index that achieves the maximum in \eqref{eq:key_ratio}, we include that \eqref{eq:xxinv_v_inf} holds:
\begin{align*}   \frac{\Abs{\vx_i^\top \vv - \vx_i^\top \vu_i}}{\vx_i^\top \matp_i \vx_i} &\ge \frac{\Abs{\vx_j^\top \vv - \vx_j^\top \vu_j}}{\vx_j^\top \matp_j \vx_j} \ge \frac 2 3 \Abs{\vx_j^\top \vv - \vx_j^\top \vu_j} \ge \frac 1 6\norm{\mx^\top_{S:}\vv}_\infty, \\
\frac{\Abs{\vx_i^\top \vv - \vx_i^\top \vu_i}}{\vx_i^\top \matp_i \vx_i} &\le 4\Abs{\vx_i^\top \vv} + 4\Abs{\vx_i^\top \vu_i} \le 6\norm{\mx^\top_{S:}\vv}_\infty. \end{align*}
\end{proof}

We again state a corollary due to Lemma~\ref{lem:fix_param_indep_model}.

\begin{corollary}
    \label{cor:noise_linf_bound_indep}
Let $\delta \in (0, \half)$, $n = \Omega(k\log\frac d \delta)$ for an appropriate constant, and $k \in [d]$. Under Model~\ref{model:subg}, for any $S \perp \mx$ such that $|S| \le k$, and $\vv \perp \mx_{:S}$, with probability $\ge 1 - \delta$ we have
    \begin{equation}
        \label{eq:xxinv_v}
        \infnorm{\xxinv_{S\times S}\mx_{S:}^\top \vv} \le 8\min\paren{\infnorm{\mx_{S:}^\top\vv}, \twonorm{\vv}\sqrt{\frac{C\log\frac{d}{\delta}}{n}}}.
    \end{equation}
\end{corollary}
\begin{proof}
This follows from Lemmas~\ref{lem:fix_param_indep_model} and~\ref{lem:noise_linf_bound_fix}, where we note that the proof of Lemma~\ref{lem:noise_linf_bound_fix} only required the independence assumption between $\vv$ and $\mx_{:S}$ (rather than all of $\mx$).
\end{proof}

\subsection{Support identification}\label{ssec:phase2}

We proceed to analyze the support identification step in Line~\ref{line:support_id} of Algorithm~\ref{alg:IHT+OLS}.
Direct calculation shows that in Problem~\ref{prob:linfty_sparse}, for any $i \in \supp(\vths)$ and $j \in \supp(\vths)^c$, we have 
\begin{equation}\label{eq:indep_calc}
    \vx_i^\top\vy = \twonorm{\vx_i}^2\vths_i + \innp{\vx_i, \mx_{:-i}\vths_{-i}} + \vx_i^\top \vxi, \quad \vx_j^\top \vy = \innp{\vx_j, \mx\vths} + \vx_j^\top \vxi.
\end{equation}
Our goal is to separate such a pair $(i, j)$, whenever $|\vths_i| \gg \norms{\mx^\top \vxi}_\infty$ is large (intuitively, a ``heavy'' coordinate of the support). We obtain signal from $\norm{\vx_i}_2^2 \vths_i \gg \norms{\mx^\top \vxi}_\infty$, and 
the pure noise terms $|\vx_i^\top \vxi|$, $|\vx_j^\top \vxi|$ are trivially bounded by $\infnorm{\mx^\top \vxi}$. Thus, our goal is to ensure the leftover terms, $\inprod{\vx_i}{\mx_{:-i}\vths_{-i}}$ and $\innp{\vx_j, \mx\vths}$, are small in magnitude. If we can ensure this holds, then a simple thresholding operation on $\mx^\top \vy$ will separate large coordinates of $\vths$ from zero coordinates.

We now formalize this idea, crucially leveraging independence of $\mx$ and $\vths$.

\begin{lemma}
    \label{lem:large_coord_est}
Let $\delta \in (0, \half)$, $n = \Omega(k \log \frac d \delta)$ for an appropriate constant, and $k \in [d]$. Under Model~\ref{model:subg}, suppose that $L \defeq \{i \in [d] : |\vx_i^\top \vy| \ge r_\infty\}$, for 
\[r_\infty \ge \frac{\norm{\vths}_2}{\sqrt k} + 3\norm{\mx^\top \vxi}_\infty.\]
Then the following hold with probability $\ge 1 - \delta$:
\[L \subseteq \supp(\vths),\quad \norm{\vths_{L^c}}_\infty \le 4r_\infty.\]
\end{lemma}

\begin{proof}
Throughout the proof, denote $S \defeq \supp(\vths)$. We claim both of the following hold, with probability at least $1 - \delta$: $\mx$ is $(\half, 1)$-RIP, and 
\[\max\Brace{\max_{i \in S} \inprod{\vx_i}{\mx_{:-i}\vths_{-i}},\; \max_{j \in S^c} \inprod{\vx_j}{\mx\vths} } \le \frac{\norms{\vths}_2}{\sqrt k}.\]
To see that this is correct, RIP holds with probability $\ge 1 - \frac \delta 2$ by Proposition~\ref{prop:RIP_sub_gaussian}, which implies all $i \in [d]$ have $\norm{\vx_i}_2 \in [\half, \frac 3 2]$. Hence, because  $(i, j) \in S \times S^c$ satisfy $\vx_i \perp \mx_{:-i} \vths_{-i}$, and $\vx_j \perp \mx \vths$, applying Corollary~\ref{cor:cross_term_indep} with failure probability $\delta \gets \frac{\delta}{2d}$ gives the bound for large enough $n$. 

Condition on the above events holding henceforth.
Now recalling the decomposition \eqref{eq:indep_calc}, 
\begin{align*}
|\vx_j^\top \vy| \le \frac{\norm{\vths}_2}{\sqrt k} + \norm{\mx^\top \vxi}_\infty \text{ for all } j \not\in S. \end{align*}
Finally, supposing some $i \in [d]$ with  $|\vths_i| > 4r_\infty$ is not contained in $L$ gives a contradiction:
\begin{align*}
|\vx_i^\top \vy| > \half\Abs{\vths_i} -\frac{\norm{\vths}_2}{\sqrt k} - \norm{\mx^\top \vxi}_\infty  \ge r_\infty.
\end{align*}
\end{proof}
\subsection{In-support estimation}\label{ssec:phase3}
Once the large coordinates of $\vths$ are isolated, the remaining task reduces to a low-dimensional regression problem restricted to the recovered support. As our main helper lemma, we show that running ordinary least squares (OLS) on the learned support directly gives us accurate estimation in the $\ell_\infty$ norm, by crucially leveraging independence in Model~\ref{model:for_each}.

\begin{lemma}
    \label{lem:OLS_bound}
Let $\delta \in (0, \half)$, $n = \Omega(k\log \frac d \delta)$ for an appropriate constant, and $k \in [d]$. Under Models~\ref{model:for_each} and~\ref{model:subg}, suppose that $L \subseteq \supp(\vths)$, and that $L \perp \mx$. Then with probability $\ge 1 - \delta$,
\[\infnorm{\vths_{L} - [\mx^{\top}\mx]^{-1}_{L \times L} \mx^{\top}_{L:}\vy} \leq 8\norm{\mx^\top \vxi}_\infty + \frac 1 {\sqrt k}\norm{\vths}_2.\]
\end{lemma}
\begin{proof}
Because we assumed $L \subseteq S \defeq \supp(\vths)$, we can decompose $\mx \vths = \mx_{:L} \vths + \mx_{:L^c} \vths_{L^c}$, so the OLS solution also admits the following decomposition:
\[[\mx^{\top}\mx]^{-1}_{L \times L} \mx^{\top}_{L:}\vy = \vths_{L} + [\mx^{\top}\mx]^{-1}_{L \times L}\mx^{\top}_{L:}\mx_{:L^c}\vths_{L^c} + [\mx^{\top}\mx]^{-1}_{L \times L} \mx^{\top}_{L:}\vxi.\]
In the remainder of the proof we bound the last two terms on the right-hand side. First, condition on $\mx$ satisfying $(\half, k)$-RIP, with fails with probability $\frac \delta 3$ by Proposition~\ref{prop:RIP_sub_gaussian}. Next, letting $\vw \defeq \mx_{L^c} \vths_{L^c}$, we have by RIP that $\norm{\vw}_2 \le 2\norms{\vths}_2$, and $(L, \vw) \perp \mx_{:L}$ by assumption. Thus, by Corollary~\ref{cor:noise_linf_bound_indep},
\[\norm{\Brack{\mx^\top \mx}_{L \times L}^{-1}\mx_{L:}^\top\vw}_\infty \le 8 \norm{\vw}_2 \sqrt{\frac{C\log \frac d \delta}{n}} \le \frac 1 {\sqrt k}\norm{\vths}_2,\]
except with probability $\frac \delta 3$, by our choice of $n$. Finally, again by Corollary~\ref{cor:noise_linf_bound_indep},
\[\norm{\Brack{\mx^\top \mx}_{L \times L}^{-1}\mx_{L:}^\top \vxi}_\infty \le 8\norm{\mx_{L:}^\top \vxi}_\infty \le 8\norm{\mx^\top \vxi}_\infty,\]
except with probability $\frac \delta 3$, because $\vxi \perp \mx$. The conclusion follows from combining the above two displays and applying a union bound for the failure probability.
\end{proof}

Finally, we give an end-to-end guarantee for Algorithm~\ref{alg:IHT+OLS}.
\begin{theorem}[Oblivious $\ell_\infty$ sparse recovery]
    \label{thm:l_inf_norm_bound_model1}
Let $\delta \in (0, \half)$, $n = \Omega(k \log \frac d \delta)$ for an appropriate constant, and $k \in [d]$. Under Models~\ref{model:for_each} and~\ref{model:subg}, there is a universal constant $c > 0$ such that if Algorithm~\ref{alg:IHT+OLS} is called with parameters $c$, $R \ge \norm{\vths}_2$, and 
\[r \ge \sigma \defeq \norm{\mx^\top\vxi}_\infty \sqrt{\log \frac{n}{\delta}},\]
then with probability $\ge 1 - \delta$, its output satisfies
\begin{equation}
        \label{eq:l_inf_norm_bound_model1}
        \infnorm{\vth - \vths} = O\paren{r},\quad \Abs{\supp(\vth)} = O\Par{k}.
    \end{equation}
    Moreover, the algorithm runs in time $O(nd\log \frac R r + nk\log\frac{Rk}{r})$.
\end{theorem}
\begin{proof}
By Proposition~\ref{prop:RIP_sub_gaussian}, with probability $\ge 1 - \frac \delta 4$, $\mx^{(1)}, \mx^{(2)}, \mx^{(3)}$ are all $(0.14, 3k)$-RIP. Moreover, we claim that with probability $\ge 1 - \frac \delta 4$, for all $i \in [3]$, there exists a constant $A > 0$ such that
\begin{equation}\label{eq:each_obs_bound}\norm{\mx^{(i)\top}\vxi^{(i)}}_\infty \le A \sigma.\end{equation}
To see this, each of the relevant quantities is the maximum of either $\frac n 3$ or $n$ i.i.d.\ sub-Gaussian random variables.
Standard sub-Gaussian concentration (see the derivation of $(b)$ in Lemma~\ref{lem:equiv_err}) shows that with probability $\frac \delta 4$, 
\begin{equation}\label{eq:compare_upper_lower}
\norm{\mx^{(i)\top}\vxi^{(i)}}_\infty  =O\Par{ \norm{\vxi^{(i)}}_2 \sqrt{\frac{\log \frac n \delta} n}} = O\Par{\norm{\vxi}_2\sqrt{\frac{\log \frac n \delta} n}} = O\Par{\norm{\mx^\top \vxi}_\infty\sqrt{\log \frac n \delta}}.
\end{equation}

Next, we analyze the three stages of Algorithm~\ref{alg:IHT+OLS} in turn.
By Lemma~\ref{lem:IHT_l2_recover}, $\vhth$ returned on Line~\ref{line:iht} has
\begin{equation}\label{eq:stage1_bound}\norm{\vhth - \vths}_2 \le 5\sqrt{3k} A\sigma + \sqrt k r \le 10A\sqrt k r.\end{equation}
In the remainder of the proof, let $\bvth \defeq \vths - \vhth$. Note that the pairs $(\mx^{(2)}, \vr^{(2)})$ and $(\mx^{(3)}, \vr^{(3)})$ are exactly instances of Model~\ref{model:for_each} with sparsity parameter $k \gets 2k$, and the signal-noise pairs $(\bvth, \vxi^{(2)})$ and $(\bvth, \vxi^{(3)})$ respectively. Moreover, the learned support $L$ from Line~\ref{line:support_id} satisfies $L \perp \mx^{(3)}$.

Next we apply Lemma~\ref{lem:large_coord_est}. We first check its preconditions hold: indeed,
    \[3\norm{\mx^{(2)\top}\vxi^{(2)}}_\infty + \frac{\norm{\bvth}_2}{\sqrt k} \le 13 A r, \]
    so it is enough to choose $\frac 1 c \ge 13A$. Then, with probability $\ge 1 - \frac \delta 4$, Lemma~\ref{lem:large_coord_est} gives
    \[L \subseteq \supp(\bvth),\quad \norm{\bvth_{L^c}}_\infty \le \frac{4r}{c}.\]
    Finally, applying Lemma~\ref{lem:OLS_bound}, with probability $\ge 1 - \frac{\delta}{4}$, 
    \begin{equation*}
        \infnorm{\bvth_{L} - [\mx^{(3)\top}\mx^{(3)}]^{-1}_{L \times L} \mx^{(3)\top}_{L:}\vr^{(3)}}\le 8C'\sigma + 10Ar \le (10A + 8C')r. 
    \end{equation*}
    Summing the above two displays thus yields for the output $\vth$,
    \[\norm{\vth - \vths}_\infty = \norm{\vth - (\bvth + \vhth)}_\infty \le \norm{\bvth_{L^c}}_\infty + \norm{\bvth_L - (\vth - \vhth)}_\infty = O(r). \]
    Finally, to prove the sparsity claim, it suffices to note that both the support of $\vhth$ output by $\IHT$, and the learned support $L$, have cardinality at most $O(k)$.

    We conclude by discussing the runtime. All steps in Algorithm~\ref{alg:IHT+OLS} clearly take time at most $O(nd\log \frac R r)$
    except for the OLS step, Line~\ref{line:ols}. Whenever Lemma~\ref{lem:large_coord_est} holds, $|L| \le k$, so under our stated sample complexity, the matrix $\ma \defeq [\mx^{(3)\top}\mx^{(3)}]_{L \times L}$ is $O(1)$-well-conditioned. Our goal is to compute $\ma^{-1} \vb$ up to $O(r)$ in $\ell_\infty$ error, where $\vb = \mx^{(3)\top}\vr^{(3)}$. Further, 
    \[\norm{\vb}_2 \le \norm{\mx^{(3)\top}\mx^{(3)} \vhth}_2 + \norm{\mx^{(3)\top}\vxi^{(3)}}_2 \le 2\norm{\vhth}_2 + Ar \le 2R +  20A\sqrt k r + Ar,\]
    where we used the triangle inequality, RIP and the assumption \eqref{eq:each_obs_bound}, and the recovery guarantee \eqref{eq:stage1_bound} respectively in the three inequalities. Standard bounds on gradient descent, e.g., Theorem 2.1.15, \cite{Nesterov03}, then show that $\log(\frac {kR} r)$ iterations yield a sufficiently-accurate solution. Each iteration requires $O(1)$ matrix-vector multiplications through $\mx^{(3)}_{:L}$, each taking $O(nk)$ time.

\end{proof}

The polylogarithmic overhead in the error of Theorem~\ref{thm:l_inf_norm_bound_model1} is due to our use of a small amount of adaptivity, leading to a requirement of the form \eqref{eq:each_obs_bound} to compare each phase's error to the overall observations. Examining the proof, it is clear we can replace the error bound with
\[O\Par{\max_{i \in [3]} \norm{\mx^{(i)\top}\vxi^{(i)}}_\infty}.\]
In general, as long as each individual $\norms{\mx^{(i)\top}\vxi^{(i)}}_\infty$ is comparable to $\norms{\mx^\top \vxi}_\infty$, then we have the tight requirement $r = \Omega(\norms{\mx^\top \vxi}_\infty)$. We also point out that for certain popular random matrix models, we can obtain this improvement directly in some parameter regimes. For example, if $\mx$ is entrywise a (scaled) Gaussian, and $\delta = \poly(\frac 1 k)$, then using tight bounds for the maximum of $k$ Gaussians in place of \eqref{eq:compare_upper_lower} again yields a tight bound $r = \Omega(\norms{\mx^\top \vxi}_\infty)$.

We also mention here that it is simple to boost any improper hypothesis for Problem~\ref{prob:linfty_sparse} into a proper (i.e., exactly $k$-sparse) hypothesis, at a constant factor overhead in the recovery guarantee.

\begin{lemma}\label{lem:proper}
Let $\vths \in \R^d$ have $\nnz(\vths) \le k$. Then for any $\vth \in \R^d$ and $p \in \R_{\ge 1} \cup \{\infty\}$,
\[\norm{H_k(\vth) - \vths}_p \le 2\norm{\vth - \vths}_p.\]
\end{lemma}
\begin{proof}
Because $H_k(\vth)$ is the projection of $\vth$ onto the set of $k$-sparse vectors under the $\ell_p$ norm,
\[\norm{H_k(\vth) - \vths}_p \le \norm{H_k(\vth) - \vth}_p + \norm{\vth - \vths}_p \le 2\norm{\vth - \vths}_p.\]
\end{proof}

\subsection{Expanding the \texorpdfstring{$r$}{r} range in Theorem~\ref{thm:l_inf_norm_bound_model1}}\label{ssec:expand_r}

A potential shortcoming of Theorem~\ref{thm:l_inf_norm_bound_model1} is that it requires the lower bound $r \gtrsim \norms{\mx^\top \vxi}_\infty$ to hold, and its guarantees do not formally apply otherwise. In this section, we give a simple black-box reduction that lifts this requirement, yielding error guarantee $\lesssim r + \norms{\mx^\top \vxi}_\infty$ for \emph{any} $r > 0$. This matches the form of our $\IHT$ bounds in Lemma~\ref{lem:IHT_l2_recover} and Theorem~\ref{thm:linf_iht}.

Our reduction (Algorithm~\ref{alg:oblivious-general}) maintains a threshold $\rho$, and applies Theorem~\ref{thm:l_inf_norm_bound_model1} in phases. In each phase it then uses holdout samples to check whether a certain test passes. Finally, it outputs the estimated parameters obtained on the last phase that passed.
\begin{algorithm}[ht]
\DontPrintSemicolon
    \caption{$\OSRR(\mx, \vy, k, R, r, c')$}
    \label{alg:oblivious-general}
     $\vth^{(0)} \gets \vzero_d$;
    \If{$r \ge R$}{
    \Return {$\vth^{(0)}$}
    }
    $T \gets \lceil\log_2 \frac{R}{r}\rceil$\;
    $c \gets $ constant in the statement of Theorem~\ref{thm:l_inf_norm_bound_model1}\;
    $\rho^{(0)} \gets R$\;
    $(\mx, \vy) \gets \sqrt{2T}(\mx,\vy)$\;\label{line:scaling_rip}
    Evenly divide observations $(\mx, \vy)$ into $\mx = \begin{bmatrix}
        \mx^{(1)}\\
        \vdots\\
        \mx^{(2T)}\\
    \end{bmatrix}, \vy = \begin{bmatrix}
        \vy^{(1)}\\
        \vdots\\
        \vy^{(2T)}
    \end{bmatrix}$\;    
    \For {$t = 0, 1, \ldots, T-1$}{
    $\rho^{(t + 1)} \gets \frac {\rho^{(t)}} 2$\;
    $\vth^{(t+1)} \gets H_k(\OSR(\mx^{(2t+1)}, \vy^{(2t+1)}, k, R, \rho^{(t + 1)}, c))$\;
    \If{$\norms{(\mx^{(2t+2)})^\top(\vy^{(2t+2)} - \mx^{(2t+2)}\vth^{(t+1)})}_\infty>\frac 1 {c'}\rho^{(t + 1)}$\label{line:test_holdout}}{
    \Return {$\vth \gets \vth^{(t)}$}
    }
    $t\gets t+1$\;
    }
    \Return{$\vth \gets \vth^{(T)}$}
\end{algorithm}

We require one helper result in the analysis of Algorithm~\ref{alg:oblivious-general}.

\begin{lemma}\label{lem:preserve_norm_indep}
Let $\delta \in (0, 1)$, and let $\vv \in \R^d$ be fixed with $\nnz(\vv) \le k$. If $n = \Omega(k\log \frac d \delta)$ for an appropriate constant, then under Model~\ref{model:for_each}, with probability $\ge 1 - \delta$,
\[\frac{\norm{\mx^\top \mx \vv}_\infty}{\norm{\vv}_\infty} \in \Brack{\half, 2}.\]
\end{lemma}
\begin{proof}
The argument follows the same structure as Lemma~\ref{lem:noise_linf_bound_fix}, but is simpler.
By Proposition~\ref{prop:RIP_sub_gaussian}, with probability at least $1-\frac \delta  2$, the matrix $\mx$ is $(\frac 1 4,k)$-RIP. Let $S=\supp(\vv)$ with $|S|\le k$. 

Under $(\frac 1 4,k)$-RIP, $\|\mx\vv\|_2^2\le 2\|\vv\|_2^2\le 2k\|\vv\|_\infty^2$.
For any $i\notin S$, since $\vx_i\perp \mx_{:S}$, by Lemma~\ref{lem:hoeffding}, 
\begin{equation}\label{eq:offsupport_bound}
\prob{ \Abs{\Brack{\mx^\top \mx \vv}_i} > \frac{1}{4}\infnorm{\vv} } \le 2\exp\paren{-\frac{n}{2Ck}} \le \frac{\delta}{2d}, \text{ for all } i \not\in S,
\end{equation}
where $C$ is as in Model~\ref{model:subg},
and the last inequality takes $n = \Omega(k\log\tfrac{d}{\delta})$ sufficiently large. 

Next, fix $j\in S$ and write $T_j = S \backslash \cbra{j}$. Then
\begin{equation*}
    \innp{\vx_j, \mx_{:S}\vv} = \twonorm{\vx_j}^2\vv_1 + \innp{\vx_j, \mx_{:T_j}\vv_{-j}}.
\end{equation*}
Since $\vx_j\perp\mx_{:T_j}$, the second term satisfies the same sub-Gaussian bound as above, hence
\begin{equation}\label{eq:onsupport_bound}
\prob{ \Abs{\innp{\vx_j, \mx_{:T_j}\vv_{-j}}} > \frac{1}{4}\infnorm{\vv} } \le \frac{\delta}{2d}, \text{ for all } j \in S.
\end{equation}
By a union bound over all $d$ events in \eqref{eq:offsupport_bound}, \eqref{eq:onsupport_bound}, and using $\twonorm{\vx_j}^2 \le \frac 5 4$ for all $j \in S$, we then have
\[
\max_{i\not\in S}\Abs{\Brack{\mx^\top \mx \vv}_i} \le \frac{1}{4}\infnorm{\vv}, \quad \max_{j \in S} \Abs{\Brack{\mx^\top \mx \vv}_j} \le \frac{3}{2}\infnorm{\vv},
\]
which yields the claimed upper bound. 
For the lower bound, let $j^\star\in S$ satisfy $|\vv_{j^\star}|=\|\vv\|_\infty$. Using $\|\vx_{j^\star}\|_2^2\ge \frac 3 4$ and the deviation bound in \eqref{eq:onsupport_bound},
\[
\Abs{\Brack{\mx^\top \mx \vv}_{j^\star}} \ge \frac{3}{4}\Abs{\vv_{j^\star}} - \frac{1}{4}\infnorm{\vv} = \frac{1}{2}\infnorm{\vv}.
\]

\end{proof}

We now prove that Algorithm~\ref{alg:obs_gen} obtains an  $\ell_\infty$ recovery guarantee of the form $\approx \norms{\mx^\top \vxi}_\infty + r$. For simplicity, we assume the error bound in Theorem~\ref{thm:l_inf_norm_bound_model1} only applies in the regime
\[r \ge \norm{\mx^\top \vxi}_\infty\sqrt{\log \frac n \delta},\]
although we inherit similar improvements under certain models (e.g., Gaussian measurements) where Theorem~\ref{thm:l_inf_norm_bound_model1} offers a tighter error bound, see the discussion after the proof.

\begin{corollary}
    \label{cor:l_inf_norm_bound_model1-general}
Let $\delta \in (0, \half)$, $n = \Omega(kT \log \frac {dT} \delta)$ for an appropriate constant, and $k \in [d]$, where $T = O(\log \frac R r)$ is as in Algorithm~\ref{alg:oblivious-general}. Under Models~\ref{model:for_each} and~\ref{model:subg}, there is a universal constant $c' > 0$ such that if Algorithm~\ref{alg:oblivious-general} is called with parameters $c'$, $R \ge \norm{\vths}_2$, and $r > 0$, 
then with probability $\ge 1 - \delta$, its output satisfies
\begin{equation}
    \label{eq:l_inf_norm_bound_model1_expand_r}
    \infnorm{\vth - \vths} = O\paren{r + \norms{\mx^\top \vxi}_\infty\sqrt{\log \Par{\frac n \delta}\log\Par{\frac R r}}}.
\end{equation}
\end{corollary}

\begin{proof}
We begin by noting that we scaled all of the divided measurements, $(\mx^{(i)}, \vy^{(i)})$ for $i \in [2T]$, by $\sqrt{2T}$, to ensure that they fall in the setting of Model~\ref{model:subg} after adjusting for the smaller sample size $\frac{n}{2T}$.
Let $C > 0$ be such that Theorem~\ref{thm:l_inf_norm_bound_model1} and Lemma~\ref{lem:proper} guarantee that 
\begin{equation}\label{eq:good_run}
\rho^{(t)} \ge \sig \defeq \norm{\mx^\top \vxi}_\infty \sqrt{\log \Par{\frac {n} \delta}\log\Par{\frac R r}}
\implies
\norm{\vth^{(t)} - \vths}_\infty \le C\rho^{(t)},\end{equation}
and such that in every iteration $t \in [T]$,
\begin{equation}\label{eq:noise_compare}
\begin{gathered}
\norm{\mx^{(2t)\top}\vxi^{(2t)}}_\infty \le C\sigma,\quad \frac{\norm{\mx^{(2t)\top} \mx^{(2t)}(\vth^{(2t)} - \vths)}_\infty}{\norm{\vth^{(2t)} - \vths}_\infty} \in \Brack{\half, 2}.
\end{gathered}
\end{equation}
Note that \eqref{eq:good_run} holds with probability $\ge 1 - \frac \delta 3$ by taking a union bound over Theorem~\ref{thm:l_inf_norm_bound_model1} for $T \le n$ rounds. The first bound in \eqref{eq:noise_compare} also holds with probability $\ge 1 - \frac \delta 3$ by a union bound for an appropriate $C$, via \eqref{eq:compare_upper_lower}. Finally, the second bound in \eqref{eq:noise_compare} holds with probability $\ge 1 - \frac \delta 3$ by Lemma~\ref{lem:preserve_norm_indep} and $\nnz(\vth^{(t)} - \vths) \le 2k$. 
We henceforth condition on both events, and let $\frac 1 {c'} \ge 3C$.

We next claim that in every round where $\rho^{(t + 1)} \ge \sigma$, the test in Line~\ref{line:test_holdout} will pass. This is because
    \begin{align*}
        \norm{\mx^{(2t+2)\top}(\vy^{(2t + 2)} - \mx^{(2t + 2)} \vth^{(t + 1)})}_\infty & \le \norm{\mx^{(2t+2)\top}\mx^{(2t+2)} (\vths - \vth^{(t + 1)})}_\infty + \norm{\mx^{(2t+2)\top} \vxi^{(2t + 2)}}_\infty \\
        & \le 2 \norm{\vths - \vth^{(t + 1)}}_\infty + \norm{\mx^{(2t+2)\top} \vxi^{(2t + 2)}}_\infty \\
        & \le 3C\rho^{(t + 1)} \le \frac 1 {c'} \rho^{(t + 1)}.
    \end{align*}
    Therefore, the last round where the test passes will satisfy $\rho^{(t + 1)} \le \max(r, 2\sigma)$. On this round, 
    \begin{align*}
        \norm{\vth^{(t)} - \vths}_\infty &\le 2\norm{\mx^{(2t)\top} \mx^{(2t)} (\vths - \vth^{(t)})}_\infty \\
        & \le 2 \norm{\mx^{(2t)\top}(\vy^{(2t)} - \mx^{(2t)} \vth^{(t)})}_\infty + 2 \norm{\mx^{(2t)\top} \vxi^{(2t)}}_\infty \\
        & \le \frac{\max(2\rho, 4\sigma)}{c'} + C\sigma = O\left(r+\norm{\mx^\top \vxi}_\infty\sqrt{\log\Par{ \frac n \delta}\log\Par{\frac R r}}\right).
    \end{align*}
\end{proof}

%% file: tex_file/model3.tex
\section{Adaptive \texorpdfstring{$\ell_\infty$}{L-inf} Sparse Recovery}
\label{sec:adv_noise}

In this section, we study fundamental limits and algorithms for Problem~\ref{prob:linfty_sparse} under Model~\ref{model:adv_noise}. 

We first establish that $n = \Omega(k^2 \log d)$ sub-Gaussian measurements suffice to solve adaptive $\ell_\infty$ sparse recovery. We in fact prove our upper bound result under a more general regularity condition, an $\ell_\infty$-norm variant of RIP, defined and studied in Section~\ref{ssec:linfty_rip}. Under this condition, we give our main algorithm for solving Problem~\ref{prob:linfty_sparse} in the adaptive model in Section~\ref{ssec:iht_linfty}.

In Section~\ref{ssec:linfty_lower}, we then construct a nearly-matching lower bound, which shows that sample sizes of order $n = O(k^2)$ can be fooled by adversarial instances. This rules out the possibility for any algorithm to solve the adaptive variant of Problem~\ref{prob:linfty_sparse} in this sample complexity regime. 

\subsection{\texorpdfstring{$\ell_\infty$}{L-inf}-RIP}\label{ssec:linfty_rip}
Motivated by the classical RIP (Definition~\ref{def:RIP}), we introduce an $\ell_\infty$ analog that controls the Gram matrix $\mx^\top\mx$ under a different norm. While RIP bounds the $\ell_2$ operator norm of $\mx^\top\mx-\id_d$ on sparse submatrices, $\ell_\infty$-RIP instead bounds the $\ell_\infty$ operator norm of the same restrictions. 

\begin{definition}[$\ell_\infty$-RIP]
    \label{def:linf_RIP}
    Let $(\eps, s) \in (0, 1) \times [d]$.
    We say $\mx \in \bbR^{n\times d}$ is $(\epsilon, s)$-$\ell_\infty$-RIP, if for all $S \subseteq [d]$ with $|S| \le s$, 
    \begin{equation}\label{eq:linf_rip_def}
        \norm{[\mx^\top\mx - \id_d]_{S\times S}}_{\infty \to \infty} \le \eps.
    \end{equation}
\end{definition}

The following observation shows that $\ell_\infty$-RIP implies $\ell_2$-RIP with the same parameters.

\begin{lemma}\label{lem:linfty_l2}
For any symmetric matrix $\mm \in \R^{d \times d}$, we have $\norms{\mm}_{2 \to 2} \le \norms{\mm}_{\infty \to \infty}$.
\end{lemma}
\begin{proof}
This follows because $\lam = \norms{\mm}_{2 \to 2}$ is witnessed by some eigenvector $\vv$ with either $\mm \vv = \lam \vv$ or $\mm \vv = -\lam \vv$, both of which witness the same blowup in the $\ell_\infty$ norm.
\end{proof}

The $\ell_\infty$-RIP condition is also closely related to pairwise incoherence (PI, Definition~\ref{def:PI}). Our next lemma shows that $\frac \eps s$-PI implies $(\eps,s)$-$\ell_\infty$-RIP. The converse, however, does not hold: $\ell_\infty$-RIP permits structured correlations that are not excluded by the pairwise bounds of PI.

\begin{lemma}[PI implies $\ell_\infty$-RIP]
    \label{lem:pi_to_linf_RIP}
    For any $(\eps, s) \in (0, 1) \times [d]$, if $\mx\in \R^{n\times d}$ is $\frac{\eps}{s}$-PI, then $\mx$ is also $(\eps, s)$-$\ell_\infty$-RIP. Moreover, the converse statement is false.
\end{lemma}
\begin{proof}
As remarked on in Section~\ref{ssec:notation}, the $\ell_\infty\to\ell_\infty$ norm of a matrix is the largest $\ell_1$ norm of any of its rows, so for any $|S| \le s$, the bound \eqref{eq:linf_rip_def} is immediate using $\frac \eps s$-PI. To disprove the converse statement, take $\mx = \id_d + \eps \ve_i\ve_j^\top$ for any $i\ne j$ and $s \ge 2$.
\end{proof}

Given Lemma~\ref{lem:pi_to_linf_RIP} and Proposition~\ref{prop:pi}, sub-Gaussian ensembles are $\ell_\infty$-RIP with high probability. We show that directly applying a union bound gives a slightly tighter guarantee.

\begin{lemma}
    \label{lem:subg_linf_rip}
    Let $(\delta, \eps) \in (0, \half)^2$ and $s \in [d]$. Under Model~\ref{model:subg}, if $n = \Omega (\frac{s^2}{\eps^2} (\log \tfrac{d}{s} + \log\tfrac{1}{\delta}))$ for an appropriate constant, then with probability $1- \delta$, $\mx$ satisfies $(\eps,s)$-$\ell_\infty$-RIP.
\end{lemma}
\begin{proof}
Write $\mg := \mx^\top\mx-\id_d$. We control the diagonal and off-diagonal terms separately. 

For each $i$, $\mg_{ii}=\|\vx_i\|_2^2-1$ is a centered $\chi^2$ random variable, so Lemma 1, \cite{LaurentM00} shows
\[
\prob{|\mg_{ii}|\ge \frac \eps 2}\le 2\exp\Par{-\Omega(\eps n)}.
\]
A union bound over $i\in[d]$ yields $\max_i|\mg_{ii}|\le \frac \eps 2$ with probability
at least $1-\frac \delta 3$ provided $n = \Omega(\frac 1 \eps \log(\frac d \delta))$, which falls within our parameter regime.
Next, by Proposition~\ref{prop:RIP_sub_gaussian}, the event $\norm{\vx_i}_2 \le 2$ for all $i \in [d]$ occurs with probability $\ge 1 - \frac \delta 3$.  Condition on both events henceforth.

Finally, for the off-diagonal terms, we first fix a subset $S\subset[d]$ with $|S|=s$, $i\in S$, and $\vv\in\{\pm1\}^S$. Conditioning on $\vx_i$, 
the $\{\vx_i^\top\vx_j:j\in S\setminus\{i\}\}$ are i.i.d.\ $\subg(0, \frac{2C}n)$ random variables. Hence,
\[
\sum_{j\in S\setminus\{i\}} \vv_j\,\vx_i^\top\vx_j
\sim \subg\Par{0, \frac{2Cs}{n}},
\]
by Fact~\ref{fact:composition}, and therefore by Lemma~\ref{lem:hoeffding},
\begin{equation}\label{eq:to_union}
\prob{\Abs{\sum_{j\in S \setminus \{i\}} \vv_j\,\vx_i^\top\vx_j} \ge \frac \eps 2}
= 2\exp\Par{-\Omega\Par{\frac{\eps^2 n}{s}}}.
\end{equation}
Note that if the event in \eqref{eq:to_union} holds, then combined with the diagonal, $|\ve_i^\top \mg_{S \times S} \vv| \le \eps$.
Now to obtain $\ell_\infty$ RIP, we must show
\[
\max_{\substack{S \subseteq [d]: |S| \le s \\ i \in S \\ \vv \in \{\pm 1\}^S }} |\ve_i^\top \mg_{S\times S}\vv|\le \eps. \]
for all of $N$ certificates $(S, i, \vv)$.
The number of events is at most
\(
N \le s\,2^s\binom{d}{k}
\),
with $\log N = \Theta(s\log(\frac d s))$. Setting the right-hand side of \eqref{eq:to_union} to $\frac \delta {3N}$ and solving for $n$ gives the claim.
\end{proof}
If we instead combine Lemma~\ref{lem:pi_to_linf_RIP} with Proposition~\ref{prop:pi} to analyze sub-Gaussian ensembles, the resulting argument yields the sufficient condition $n \gtrsim \tfrac{s^2}{\eps^2}\log \frac d \delta$, which is slightly worse than Lemma~\ref{lem:subg_linf_rip}. This gap in the logarithmic dependence can be viewed as direct evidence that $\mathrm{PI}$ imposes a more restrictive structural requirement than $\ell_\infty$--RIP.

A natural question is whether the $\approx s^2$ sample complexity established in Lemma~\ref{lem:subg_linf_rip} for sub-Gaussian ensembles can be further improved by alternative matrix ensembles. Unfortunately, as we shortly demonstrate, this dependence is essentially tight: for any matrix with suitably-normalized columns, $(\eps,s)$-$\ell_\infty$-RIP cannot hold when $n=o(s^2)$. Consequently, sub-Gaussian ensembles achieve the optimal sample complexity for constant $\eps$, up to logarithmic factors. 

To prove this argument, we first recall a classic inequality known as the Welch bounds, which lower bounds averaged correlations among a set of unit vectors. 
\begin{lemma}[\cite{welch1974lower}]
    \label{lem:welch_bound}
    Let $\cbra{\vv_i}_{i \in [d]}$ be a set of unit vectors in $\R^n$. Then,
    \begin{equation*}
        \frac{1}{d^2}\sum_{(i, j) \in [d] \times [d]} \innp{\vv_i, \vv_j}^2 \ge \frac{1}{n}.
    \end{equation*}
\end{lemma}
Then we prove our tightness result as follows.
\begin{lemma}
    \label{lem:LB_inf_RIP}
    For any $\eps \in (0, \half)$, $s \ge 2$, and $d \ge \frac{s^3}{\eps^2}$, if $\mx$ satisfies $(\eps, s)$-$\ell_\infty$-RIP, then
    $n \ge \frac{s^2}{144\eps^2}$.
\end{lemma}
\begin{proof}
Throughout the proof, we use Lemma~\ref{lem:linfty_l2} which shows $\mx$ also satisfies $(\eps, s)$-$\ell_2$-RIP, so $\norm{\vx_i}_2^2 \in [1 - \eps, 1 + \eps] \subset [\half, \frac 3 2]$ for all $i \in [d]$. Now applying Lemma~\ref{lem:welch_bound} to the $\{\vx_i \twonorm{\vx_i}^{-1}\}_{i \in [d]}$, 
\begin{equation*}
\sum_{(i,j)\in[d]\times[d]} \frac{|\langle \vx_i,\vx_j\rangle|^2}
{\|\vx_i\|_2^2\|\vx_j\|_2^2} \ge \frac{d^2}{n}.
\end{equation*}
Since all $\twonorm{\vx_i}^2 \ge \half$, it follows that $\sum_{(i, j) \in [d] \times [d]} \Abs{\innp{\vx_i, \vx_j}}^2 \ge \frac{d^2}{4n}$.

Next, define $S_i := \sum_{j\in[d]} |\langle \vx_i,\vx_j\rangle|^2$.
Averaging over $i$ yields $\frac{1}{d}\sum_{i\in[d]} S_i \ge \frac{d}{4n}$. Hence there exists an index $i_0$ such that $S_{i_0} \ge \frac{d}{4n}$. We define an index set $T$ through a thresholding rule: 
\[T := \cbra{j\in[d] : \Abs{\innp{\vx_{i_0},\vx_j}} \ge \frac{1}{2\sqrt{2n}}}.\]
Moreover, the Cauchy-Schwarz inequality gives a trivial bound for pairwise incoherence: $|\langle \vx_{i_0},\vx_j\rangle|
\le \|\vx_{i_0}\|_2\|\vx_j\|_2 \le \frac 3 2$, so 
we have
\begin{equation*}
S_{i_0} \le |T|\cdot \frac{9}{4} + (d-|T|)\cdot \frac{1}{8n}. 
\end{equation*}
Combining with $S_{i_0} \ge \frac{d}{4n}$ yields $|T| \ge \frac{d}{18n}$.
If $n \ge \frac{s^2}{18\eps^2}$, the desired claim already holds. Otherwise, $|T| \ge s$, since $d \ge \frac{s^3}{\eps^2}$. We now select a subset $S \subseteq T$ with $|S|=s$ and $i_0 \in S$ and obtain that 
\begin{equation*}
\norm{[\mx^\top\mx - \id]_{S\times S}}_{\infty\to\infty} \ge \sum_{j\in S,\,j\neq i_0} \Abs{\innp{ \vx_{i_0},\vx_j}}- |\twonorm{\vx_{i_0}}^2 - 1|.
\end{equation*}
Using the definition of $S \subseteq T$ and that
$|\|\vx_{i_0}\|_2^2 - 1| \le \eps$, we obtain
\begin{equation*}
    \eps \ge \norm{[\mx^\top\mx - \id]_{S\times S}}_{\infty\to\infty} \ge \frac{s-1}{2\sqrt{2n}} - \half \ge \frac{s}{6\sqrt{n}} - \eps.
\end{equation*}
Now, rearranging $2\eps \ge \frac{s}{6\sqrt n}$ gives the claim.
\end{proof}

\subsection{IHT under \texorpdfstring{$\ell_\infty$}{L-inf}-RIP}\label{ssec:iht_linfty}
We next show that under Definition~\ref{def:linf_RIP}, IHT (Algorithm~\ref{alg:IHT}) solves Problem~\ref{prob:linfty_sparse} even in an adaptive setting. The proof roadmap is similar to the classic $\ell_2$ case, cf.\ \cite{price2021}, where the RIP condition is used to argue contraction on a potential based on the $\ell_2$ norm of the residual, and then the sparsity of $\vths$ bounds the lossiness of the top-$k$ selection step with respect to the potential. 

We start with Lemma~\ref{lem:half_step_contract}, which shows that before the top-$k$ truncation step, the $\ell_\infty$ error of our estimation contracts by a constant factor in each iteration. 

\begin{lemma}
    \label{lem:half_step_contract}
    Following the notation of Algorithm~\ref{alg:IHT}, for all $t \in [0, T-1]$, define
    \begin{equation*}
        \vth^{(t+0.5)} = \vth^{(t)} + \mx^\top (\vy - \mx \vth^{(t)}).
    \end{equation*}
    Then in the setting of Problem~\ref{prob:linfty_sparse}, if $\mx$ is $(\eps,2k)$-$\ell_\infty$-RIP, we have 
    \begin{equation*}
        \infnorm{\vth^{(t+0.5)} - \vths} \leq \eps\infnorm{\vth^{(t)} - \vths} + \infnorm{\mx^\top \vxi}.
    \end{equation*}
\end{lemma}
\begin{proof}
    Rearranging the definition of $\vth^{(t+0.5)}$ yields 
    \begin{align*}
        \vth^{(t+0.5)} - \vths = (\id - \mx^\top\mx) (\vth^{(t)} - \vths) + \mx^\top \vxi.
    \end{align*}
    Thus, letting $S \defeq \supp(\vths) \cup \supp(\vth^{(t)})$ with $|S| \le 2k$,
    \begin{align*}
    \infnorm{\vth^{(t+0.5)} - \vths} &\le \infnorm{\Brack{\id - \mx^\top\mx}_{S \times S} (\vth^{(t)} - \vths)} + \infnorm{\mx^\top \vxi}.\\
    &\leq \norm{\Brack{\id - \mx^\top \mx}_{S \times S}}_{\infty \to \infty} \infnorm{\vth^{(t)} - \vths} + \infnorm{\mx^\top \vxi}\\
    &\leq \eps\infnorm{\vth^{(t)} - \vths} + \infnorm{\mx^\top \vxi}.
    \end{align*}
\end{proof}
Combining Lemma~\ref{lem:half_step_contract} with our earlier observation in Lemma~\ref{lem:proper}, we conclude that the $\ell_\infty$-norm error of IHT iterates admits an exponential decay toward $\infnorm{\mx^\top \vxi}$, as $\vth^{(t + 1)} = T_k(\vth^{(t + 0.5)})$.

\begin{corollary}
\label{cor:one_step_contract}
In the setting of Lemma~\ref{lem:half_step_contract}, 
\begin{equation*}
\infnorm{\vth^{(t+1)} - \vths} \leq 2\eps\infnorm{\vth^{(t)} - \vths} + 2\infnorm{\mx^\top \vxi}.
\end{equation*}
\end{corollary}

With this one-step contraction, we finally give our $\ell_\infty$ sparse recovery guarantee for IHT. 

\begin{theorem}[Adaptive $\ell_\infty$ sparse recovery]\label{thm:linf_iht}
Let $\delta \in (0, \half)$, $n = \Omega(k^2 \log \frac d \delta)$ for an appropriate constant, and $k \in [d]$. Under Models~\ref{model:adv_noise} and~\ref{model:subg}, if Algorithm~\ref{alg:IHT} is called with parameters $R \ge \norms{\vths}_\infty$ and $r > 0$, then with probability $\ge 1-\delta$, its output satisfies
\[\infnorm{\vth - \vths} \le r + 2\norm{\mx^\top \vxi}_\infty.\]
Moreover, the algorithm runs in time $O(nd\log \frac R r)$.
\end{theorem}
\begin{proof}
    By Lemma~\ref{lem:subg_linf_rip}, with probability $1-\delta$, $\mx$ is $(\frac 1 4, 2k)$-$\ell_\infty$-RIP.
    Now, by telescoping Corollary~\ref{cor:one_step_contract}, 
    \begin{align*}
        \infnorm{\vth - \vths} &\leq \paren{\frac{2}{4}}^T \infnorm{\vths} + 2\infnorm{\mx^\top \vxi} \cdot \sum_{t=0}^\infty \paren{\frac{2}{4}}^t \leq \paren{\half}^T R + 2\norm{\mx^\top \vxi}_\infty.
    \end{align*}
    Taking $T \ge  \log_2 \frac{R}{r}$, we obtain the desired claim, and the runtime is immediate.
\end{proof}

As noted, Theorem~\ref{thm:linf_iht} extends beyond Model~\ref{model:subg}, as it only uses the sufficient condition of $\ell_\infty$-RIP.
\subsection{Lower bound}\label{ssec:linfty_lower}

The previous Sections~\ref{ssec:linfty_rip} and~\ref{ssec:iht_linfty} show that if we go through the sufficient condition of $\ell_\infty$-RIP, then we must take $n = \Omega(k^2)$ for an algorithm to solve Problem~\ref{prob:linfty_sparse}. However, this does not formally preclude a lower sample complexity of $n = o(k^2)$ that solves the problem under Model~\ref{model:subg} directly. The main goal of this section is to rule out such a possibility.

The key intuition behind the construction is to identify a $k$-sparse $\vv \in \R^d$ satisfying $\infnorm{\vv} \gg \sigma$ while still having $\infnorm{\mX^\top \mX \vv} = \sigma$. 
Setting $\vxi = \mX \vv$ then produces an adversarial noise vector that significantly perturbs the underlying signal yet contributes negligibly to $\mX^\top \vxi$, effectively masking the true support. 
To formalize this idea, we first establish that under suitable conditions, we can control the $\ell_\infty$ operator norm of the \emph{inverse} of a Gram matrix from Model~\ref{model:subg}.

\begin{lemma}
    \label{lem:inf_inf_lower_bound}
Let $\delta \in (0, \half)$, $\Omega(k\log \frac d \delta) = n = o(k^2)$, and let $S \subseteq [d]$ be any fixed subset such that $|S| = k = \Omega(\log \frac 1 \delta)$. Under Model~\ref{model:subg}, with probability $\ge 1 - \delta$,
    \begin{equation*}
        \norm{\Brack{\mx^\top \mx}_{S\times S}^{-1}}_{\infty \to \infty} = \Omega\paren{\frac{k}{\sqrt{n}}}.
    \end{equation*}
\end{lemma}
\begin{proof}
By Lemma~\ref{prop:RIP_sub_gaussian}, with probability $\ge 1 - \frac \delta 3$, $\mx$ is $(\half, 1)$-RIP, so $\twonorm{\vx_i}^2 \in [\half, \frac 3 2]$ for all $i \in[d]$. 

Next, define $\me \defeq \id_d - \mx^\top \mx$.
Standard inequalities (Exercise 4.7.3, \cite{vershynin2018high}) yield that for large enough $n$ as specified, denoting $\me_S \defeq \me_{S \times S}$ for short,
$ \normop{\me_{S}} \le \frac{1}{2}$,
with probability $\ge 1 - \frac \delta 3$. 

Finally, conditioned on the realization of $\vx_i$ with $\norm{\vx_i}_2^2 \in [\half, \frac 3 2]$ for a fixed $i \in S$, Fact~\ref{fact:composition} shows that for any $j \neq i$, $\me_{ij} = \vx_i^\top \vx_j \sim \subg(0, \frac {3C} {2n})$, and $\Var[\me_{ij}] = \frac 1 n \norm{\vx_i}_2^2 \ge \frac 1 {2n}$. By Lemma~\ref{lem:anticonc},
\[\E\Brack{\Abs{\me_{ij}}} \ge \frac{c_1}{\sqrt n},\quad \Var\Brack{|\me_{ij}|} \le \frac{c_2}{n},\]
for some constants $c_1$, $c_2$, where we bounded the variance using Item~\ref{item:moments} in Lemma~\ref{lem:subg_property}. Now Cantelli's inequality shows that there exists a constant $c$ such that
\begin{equation}\label{eq:anticonc_conclusion}\Pr\Brack{|\me_{ij}| \ge \frac{4c}{\sqrt n}} \ge c.\end{equation}
For the stated range of $k$, a Chernoff bound gives that at least a $\frac {2c} 3$ fraction of the entries $j$ in $S \setminus \{i\}$ satisfy the above event, with probability $\ge 1 - \frac \delta {3}$. Thus, under the assumed parameter regimes,
\[\sum_{j \in S \setminus \{i\}} |\me_{ij}| \ge \frac{8c(k-1)}{3\sqrt n} \ge \frac{ck}{\sqrt n} \implies \sum_{j \in S} |[\id_S - \me_S]_{ij}| \ge \frac{2ck}{\sqrt n} - \half \ge \frac{ck}{\sqrt n}.\]
In the remainder of the proof, condition on the above three events holding, which gives the failure probability. Next, consider the power series: 
\begin{equation*}
    [\mx^\top \mx]_{S\times S}^{-1} = (\id_S + \me_S)^{-1} = \id_S - \me_S + \sum_{t = 2}^\infty (-\me_S)^t.
\end{equation*}
Letting $\mr_S \defeq \sum_{t = 2}^\infty (-\me_S)^t$, we have 
\begin{equation*}
    \normop{\mr_S} \leq \sum_{t = 2}^\infty \normop{\me_S}^t = \frac{\normop{\me_S}^2}{1 - \normop{\me_S}} \leq 2\normop{\me_S}^2 \le \half.
\end{equation*}
The conclusion then follows from combining the above three displays, and applying Lemma~\ref{lem:linfty_l2}:
\begin{align*}
\norm{\Brack{\mx^\top \mx}^{-1}_{S \times S}}_{\infty \to \infty} \ge \norm{\id_S - \me_S}_{\infty \to \infty} - \norm{\mr_S}_{\infty \to \infty} \ge \frac{ck}{\sqrt n} - \half = \Omega\Par{\frac{k}{\sqrt n}}.
\end{align*}
\end{proof}

\begin{remark}[Comparison with Theorem 7.21 of \cite{wainwright2019high}]
\label{remark:compare_error}
We make a brief comment on the error rate achieved by the analysis of the LASSO in \cite{wainwright2019high}, in light of Lemma~\ref{lem:inf_inf_lower_bound}. In particular,
Theorem 7.21 of \cite{wainwright2019high} implies that, when
\(\vxi\sim\calN(0,\sig^2\id)\) and \(\lambda\gtrsim \sig\), the LASSO estimator
\(\vhth\) satisfies
\[
    \infnorm{\vhth-\vths}
    \le
    \infnorm{
        \Brack{\mx^\top\mx}_{S\times S}^{-1}
        \mX_{S:}^\top\vxi
    }
    +
    \lambda
    \norm{
        \Brack{\mx^\top\mx}_{S\times S}^{-1}
    }_{\infty\to\infty},
\]
where \(S=\supp(\vths)\). In Appendix~\ref{app:discuss_gaussian_noise}, we show
that the first term is comparable to \(\infnorm{\mx^\top\vxi}\) under Model~\ref{model:for_each}. Thus the first term matches the information-theoretic optimal rate derived in Appendix~\ref{app:discuss_gaussian_noise}.

However, by Lemma~\ref{lem:inf_inf_lower_bound}, the second
term can be as large as \(\Omega(\sig\frac{k}{\sqrt n}), \) regardless of the noise model.
This is strictly worse than the information-theoretic error rate \(\infnorm{\mx^\top\vxi}\) whenever
\(n=o(k^2)\). In particular, in the regime \(n=O(k\log d)\), this term is
\(\sig\sqrt{k}\) up to log factors, which is trivially satisfied by any $\ell_2$-accurate estimator (and is a $\sqrt{k}$ factor suboptimal for $\ell_\infty$ sparse recovery).
\end{remark}
\begin{corollary}
    \label{cor:sep_lemma}
In the setting of Lemma~\ref{lem:inf_inf_lower_bound}, there is a vector $\vv \in \R^d$ with $\supp(\vv) = S$, satisfying
\[\norm{\vv}_\infty = \Omega\Par{\frac{k}{\sqrt n}},\quad \norm{\mx^\top \mx \vv}_\infty = 1.\]
\end{corollary}
\begin{proof}
    By Lemma~\ref{lem:inf_inf_lower_bound}, with probability $ \ge 1 - \frac \delta 3$, there exists $\vu \in \cbra{-1, +1}^S$ such that 
    \begin{equation*}
    \infnorm{\Brack{\mx^\top\mx}_{S\times S}^{-1}\vu} \geq \frac{ck}{\sqrt n},
    \end{equation*}
    for a constant $c$. Also, henceforth condition on $\mx$ being $(\half, k)$-RIP, which occurs with probability $\ge 1 - \frac \delta 3$.    
    We take $\vv \in \R^d$ to be the vector supported on $S$, satisfying
    \[\vv_S = \Brack{\mx^\top \mx}_{S \times S}^{-1}\vu.\]
    Thus, the lower bound on $\norm{\vv}_\infty$ is trivially satisfied. For the second claim,
    \begin{equation*}
        \infnorm{\mx^\top \mx \vv} = \max\cbra{\infnorm{\mx_{:S}^\top \mx \vv},\infnorm{\mx_{:S^c}^\top \mx \vv}}.
    \end{equation*}
    Since $\supp(\vv) = S$, the first term is $1$. 
    For the second term, notice that $S \perp \mx$, so $\mx_{S^c:}^\top \perp (\mx_{:S}, \vv)$. Thus, for each $j \in S^c$, conditioned on $(\mx_{:S}, \vv)$, we have
    \[\vx_j^\top \mx_{:S}\vv_S \sim \subg\Par{0, \frac{\twonorm{\mx_S\vv_S}^2}{n}},\quad \twonorm{\mx_S\vv_S}^2 = \vu^\top\Brack{\mx^\top\mx}_{S\times S}^{-1}\vu \le 2\norm{\vu}_2^2 = 2k. \]
    Thus, applying Lemma~\ref{lem:hoeffding} and taking a union bound over all $j \in S^c$ shows $\norms{\mx^\top\mx \vv}_\infty = O(1)$ with probability $\ge 1 - \frac \delta 2$. The result follows by adjusting both bounds by a constant.
\end{proof}
\begin{theorem}[Lower bound for adaptive $\ell_\infty$ sparse recovery]
    \label{thm:hard_inst_lower_bound}
Under Models~\ref{model:adv_noise} and~\ref{model:subg}, there is no algorithm that solves Problem~\ref{prob:linfty_sparse} with probability $\ge \frac 3 4$, when $\Omega(k \log d) = n = o(k^2)$.
\end{theorem}
\begin{proof}
Fix two arbitrary sets $S \subseteq [d]$, $T \subseteq [d]$ each with size $\frac k 2$.
In this regime, Corollary~\ref{cor:sep_lemma} holds with probability $\delta = \half$, on support $S$. Let $\bvth \in \R^d$ be an arbitrary vector supported on $T$, and let $\vv$ be the vector guaranteed by Corollary~\ref{cor:sep_lemma} on support $S$, whenever it succeeds. Now let \[\vths_1 \defeq \bvth, \quad \vths_2 \defeq \bvth + \vv, \quad \vxi_1 \defeq \mx \vv, \quad \vxi_2 \defeq \vzero_n.\]
Any algorithm for solving Problem~\ref{prob:linfty_sparse} must then be able to distinguish $\vths_1$ and $\vths_2$, because 
\[\norm{\mx^\top \vxi_1}_\infty = 1,\quad \norm{\mx^\top \vxi_2}_\infty = 0,\]
and $\norms{\vths_1 - \vths_2}_\infty = \omega(1)$ in the stated parameter regime. However, this is a contradiction, because the observation models $(\mx, \vy)$ are the same under either pair $(\vths_1, \vxi_1)$ or $(\vths_2, \vxi_2)$, so no algorithm can do better than random guessing. Thus, the success probability is upper bounded by $\half + \half \cdot \half$.
\end{proof}
We note that Theorem~\ref{thm:hard_inst_lower_bound} extends to a lower bound against variable selection as well.

\begin{theorem}[Lower bound for adaptive variable selection]\label{thm:support_lb}
In the setting of Theorem~\ref{thm:hard_inst_lower_bound}, there is no algorithm that solves Problem~\ref{prob:variable_selection} with probability $\ge \frac 3 4$.
\end{theorem}
\begin{proof}
Instate the same lower bound instance as in Theorem~\ref{thm:hard_inst_lower_bound}. 
Observe that solving Problem~\ref{prob:variable_selection} would allow us to distinguish whether the support is $T$ or $S \cup T$, i.e., it would let us distinguish $\vths_1$ from $\vths_2$ whenever Corollary~\ref{cor:sep_lemma} succeeds. Also, for sufficiently large $\bvth$ entrywise, 
the signal-to-noise assumption \eqref{eq:snr_linfty} clearly holds for both of our instances.
Thus, the same contradiction holds.
\end{proof}

%% file: tex_file/model2.tex
\section{Variable Selection with Adaptive Measurements}
\label{sec:for_all}

In this section, we consider a partial model of adaptivity that sits between Models~\ref{model:for_each} and~\ref{model:adv_noise}. Intuitively, it treats the noise $\vxi$ as benign, but the signal $\vths$ as adaptive.

\begin{model}[Partially-adaptive model]\label{model:for_all}
In the \emph{partially-adaptive model} of Problems~\ref{prob:variable_selection} and~\ref{prob:linfty_sparse}, $\vxi \perp \mx$, and we make no independence assumptions on $(\vths, \vxi)$ or $(\vths, \mx)$.
\end{model}

We note that Model~\ref{model:for_each} is an instance of Model~\ref{model:for_all}, which in turn is an instance of Model~\ref{model:adv_noise}. This motivates the natural question: does the sample complexity of Problems~\ref{prob:variable_selection} and~\ref{prob:linfty_sparse} under the intermediate Model~\ref{model:for_all} scale as $\approx k$ as in Theorem~\ref{thm:l_inf_norm_bound_model1}, or as $\gtrsim k^2$ as in Theorem~\ref{thm:hard_inst_lower_bound}?

While we do not fully resolve this question, in this section, we demonstrate that nontrivial recovery guarantees are possible under Model~\ref{model:for_all}. Specifically, if we allow for adaptive measurements where we can \emph{exclude} certain coordinates of $\vths$ from observation, then $\approx k$ samples suffice to solve Problem~\ref{prob:variable_selection}. This result is shown in Theorem~\ref{thm:partial-adaptivity}. While our result works in a different (and potentially unrealistic in some scenarios) observation model than standard sparse recovery, we hope it may serve as a stepping stone towards understanding the landscape of partially-adaptive variable selection.

In Section~\ref{ssec:nofp}, we give a helper lemma that shows how to isolate a constant fraction of undiscovered large coordinates of $\vths$. We then recurs upon this procedure in Section~\ref{ssec:recurse_adapt} to prove Theorem~\ref{thm:partial-adaptivity}.

\subsection{Few false positives in thresholding}\label{ssec:nofp}

Recall that the support estimation phase of Algorithm~\ref{alg:IHT+OLS} (Line~\ref{line:support_id}) is based on thresholding $\Abs{\vx_i^\top \vy}$.
In the \emph{oblivious} model, $\mx \perp \vths$, so the interference term
$\innp{\vx_i, \mx_{:S} \vths}$ concentrates at the scale $\infbound$. This no longer applies in the \emph{adaptive} model, where $\vths$ may be chosen adversarially after observing $\mx$. Indeed, we exploit this adaptivity for our worst-case lower bound in Theorem~\ref{thm:hard_inst_lower_bound}.

Despite this worst-case construction, we show that a single round of thresholding based on $\Abs{\vx_i^\top \vy}$ still admits a meaningful guarantee in the setting of Problem~\ref{prob:variable_selection}, where all $i \in \supp(\vths)$ satisfy
 $\Abs{\vths_i} = \Omega(\infbound)$. In other words, every nonzero signal coordinate is sufficiently \emph{heavy}.
In this setting, we show that the number of falsely-selected coordinates via thresholding is $O(k)$, and that the identified coordinates contain at least a constant fraction of the total $\ell_2$ energy of $\vths$.

We remark that this proof is inspired by a similar observation in Lemma 2.1, \cite{BlasiokBKS24}.

\begin{lemma}
    \label{lem:l2_contract}
    Let $\delta \in (0, \frac 1 2)$ and assume that for a universal constant $C' > 0$,
    \[\twonorm{\vths} \le C'\sqrt k \norm{\mx^\top \vxi}_\infty.\]
    In the setting of Problem~\ref{prob:variable_selection} and Models~\ref{model:subg} and~\ref{model:for_all}, let
    \begin{equation*}
    \sfp = \cbra{i \not\in \supp(\vths) : \Abs{\vx_i^\top \vy} \geq \frac{C}{2}\infnorm{\mx^\top \vxi}}, \quad \sfn = \cbra{i \in \supp(\vths) : \Abs{\vx_i^\top \vy} < \frac{C}{2}\infnorm{\mx^\top \vxi}},
    \end{equation*}
    where $C$ is the constant from Problem~\ref{prob:variable_selection}.
    Then if $n = \Omega(k\log \frac d \delta)$ for an appropriate constant, with probability $\ge 1 - \delta$, we have
    \[|\sfp| = O(k),\quad \frac{\twonorm{\vths_{\sfn}}}{\twonorm{\vths}} \le 0.95.\]
    
\end{lemma}
\begin{proof}
    We prove the first claim through contradiction. Suppose that there is $T \subseteq \sfp$ such that 
    \begin{equation}\label{eq:false_count}|T| = \frac{64(C')^2}{C^2} \cdot k + 1,\end{equation}
    which exists whenenver $|\sfp| > \frac{64(C')^2}{C^2} \cdot k$.
    By the thresholding rule, for any $i \in T$, $\Abs{\vx_i^\top \vy} \geq \frac{C}{2}\infnorm{\mx^\top \vxi}$. Define a vector $\vv \in \R^T$ such that $\vv_i = \sign(\vx_i^\top \vy)$. Then, 
    \begin{equation*}
        \innp{\vv, \mx_{T:}^\top\vy} = \sum_{i\in T}\Abs{\vx_i^\top \vy} \geq \frac{C}{2}|T| \cdot \infnorm{\mx^\top \vxi}.
    \end{equation*}
    On the other hand, notice that  
    \begin{equation*}
    \innp{\vv, \mx_{T:}^\top\vy} = \innp{\vv, \mx_{T:}^\top\mx\vths} + \innp{\vv, \mx_{T:}^\top\vxi} = \innp{\mx_{:T}\vv, \mx\vths} + \innp{\vv, \mx_{T:}^\top\vxi}.
    \end{equation*}
    By the Cauchy-Schwarz inequality, we have 
    \begin{equation*}
    \Abs{\innp{\mx_{:T}\vv, \mx\vths}}\leq \twonorm{\mx_{:T}\vv}\cdot \twonorm{\mx\vths}, \quad \Abs{\innp{\vv, \mx_{T:}^\top\vxi}} \leq \twonorm{\vv}\cdot \twonorm{\mx^\top_{T:}\vxi} \leq |T|\cdot \infnorm{\mx^\top \vxi}.
    \end{equation*}
    For sufficiently large $n$ as stated, $\mx$ satisfies $(\frac 1 2, |T| + k)$-RIP, so
    \begin{equation*}
        \twonorm{\mx_{:T}\vv}^2 \leq 2\twonorm{\vv}^2 = 2 |T|, \quad \twonorm{\mx\vths}^2 \leq 2\twonorm{\vths}^2.
    \end{equation*}
    This yields
    \begin{equation*}
        \Abs{\innp{\vv, \mx_{T:}^\top\vy}} \leq 2 \twonorm{\vths}\sqrt{|T|} + |T|\cdot \infnorm{\mx^\top \vxi}.
    \end{equation*}
    Combining gives the contradiction, for sufficiently large $C > 4$,
    \begin{equation*}
        2 \twonorm{\vths}\sqrt{|T|} + |T|\cdot \infnorm{\mx^\top \vxi} \geq \frac{C}{2}|T|\cdot \infnorm{\mx^\top \vxi} \implies  |T| \leq \frac{64\twonorm{\vths}^2}{C^2\infnorm{\mx^\top\vxi}^2} \le \frac{64(C')^2}{C^2} \cdot k.
    \end{equation*}
    
    To prove the second claim, we use a similar strategy.
    By the thresholding rule, for any $i \in \sfn$, we know $\Abs{\vx_i^\top \vy} < \tfrac{C}{2}\infnorm{\mx^\top \vxi}$. Then, 
    \begin{equation*}
        \twonorm{\mx_{\sfn:}^\top\vy} = \sqrt{\sum_{i\in \sfn}\Abs{\vx_i^\top \vy}^2} < \frac{C}{2}\sqrt{|\sfn|} \infnorm{\mx^\top \vxi}.
    \end{equation*}
    On the other hand, we let $\stp = S\setminus \sfn$ and notice that   
    \begin{equation*}
        \twonorm{\mx_{\sfn:}^\top\vy} \geq \twonorm{\mx_{\sfn:}^\top\mx_{:\sfn}\vths_\sfn} - \twonorm{\mx_{\sfn:}^\top\mx_{:\stp}\vths_\stp} - \twonorm{\mx_{\sfn:}^\top \vxi}.
    \end{equation*}
    A direct calculation also gives, under $(\frac 1 4, k)$-RIP,
    \begin{gather*}
        \twonorm{\mx_{\sfn:}^\top\mx_{:\sfn}\vths_\sfn} \geq \frac 3 4 \twonorm{\vths_{\sfn}}, \quad \twonorm{\mx_{\sfn:}^\top\mx_{:\stp}\vths_\stp} \leq \frac 5 4 \twonorm{\vths_{\stp}},\\
        \twonorm{\mx_{\sfn:}^\top \vxi} \leq \sqrt{|\sfn|}\infnorm{\mx^\top \vxi}.
    \end{gather*}
    Therefore, we have the following lower bound: 
    \begin{equation*}
        \twonorm{\mx_{\sfn:}^\top\vy} \geq \paren{\twonorm{\vths_{\sfn}} - \twonorm{\vths_{\stp}}}  - \half \twonorm{\vths} - \sqrt{|\sfn|} \infnorm{\mx^\top\vxi}.
    \end{equation*}
        To avoid contradiction, we require
    \begin{equation*}
        \paren{\twonorm{\vths_{\sfn}} - \twonorm{\vths_{\stp}}}  - \half \twonorm{\vths} - \sqrt{|\sfn|} \infbound \leq \frac{C}{2} \sqrt{|\sfn|}\infbound,
    \end{equation*}
    which for $C'$ sufficiently large in Problem~\ref{prob:variable_selection}, is equivalent to 
    \begin{equation*}
        \alpha - \sqrt{1 - \alpha^2} \leq \half + \frac{(\frac C 2 + 1)\sqrt{|\sfn|}\infbound}{\twonorm{\vths}} \le \frac 3 5,\text{ for } \alpha \defeq \frac{\twonorm{\vths_{\sfn}}}{\twonorm{\vths}}.
    \end{equation*}
    Finally, $\alpha \in (0, 1)$ implies $\alpha \le 0.95$ as claimed, by solving a quadratic.
\end{proof}

\subsection{Recursive support estimation}\label{ssec:recurse_adapt}

We now describe an adaptive algorithm that solves Problem~\ref{prob:variable_selection} using $\approx k$ measurements. The specific form of adaptivity we use is the ability to query new observations using Algorithm~\ref{alg:obs_gen}.

\begin{algorithm}[ht]
\DontPrintSemicolon
    \caption{$\AO(n, S)$}
    \label{alg:obs_gen}
    Generate $\mx \in \R^{n\times d}$, $\vxi \in \R^n$ under Models~\ref{model:subg} and~\ref{model:for_all}\;
    $\mx_{:S} \gets \vzero_{n\times |S|}$\;
    Observe $\vy \gets \mx \vths + \vxi$\;
    \Return {$(\mx, \vy)$}
\end{algorithm}
In other words, we fix the signal $\vths$ in all adaptive calls to Algorithm~\ref{alg:obs_gen}, but each call generates a new independent $(\mx, \vxi)$ respecting the setting of Problem~\ref{prob:variable_selection}. Note that $\vths$ can arbitrarily depend on all of these measurement and noise generations under Model~\ref{model:for_all}; for example, we can then reuse these same measurements and noise on future instances of Problem~\ref{prob:variable_selection}. However, our adaptive observation model in Algorithm~\ref{alg:obs_gen} also assumes the ability to mute specified columns of $\mx$. This can be a strong assumption, and obtaining a comparable result without this access is an interesting open direction.

We give our variable selection algorithm in Algorithm~\ref{alg:adap_est}, which assumes access to Algorithm~\ref{alg:obs_gen}.

\begin{algorithm}[ht]
\DontPrintSemicolon
    \caption{$\AOSR(k, n, d, N, R, r_2, r_\infty, \delta)$}
    \label{alg:adap_est}
    $\mx^{(0)}, \vy^{(0)} \gets \AO(\frac n 3, \emptyset)$\;
    $\vhth_0 \gets \IHT(\mx^{(0)}, \vy^{(0)}, k, R, r_2)$\;
    $T_0 \gets \supp(\vhth_0)$\;
    \For {$i \in [N]$}{
    $\mx^{(i)}, \vy^{(i)} \gets \AO(\frac n {3N}, T_{i-1})$\;
    $T_i \gets T_{i-1} \cup\{j: |\vx^{(i)\top}_j\vy^{(i)}_r| \geq r_\infty\}$\;
    }
    $S \gets T_N$\;
    $\mx^{(N + 1)}, \vy^{(N + 1)} \gets \AO\paren{\frac n {3}, S^c}$\;
    $\vth \gets H_k\paren{[\mx^{(N + 1)\top}\mx^{(N + 1)}]_{S\times S}^{-1}\mx^{(N + 1)\top}_{S:} \vy^{(N + 1)}}$\;
    \Return {$\vth$}
\end{algorithm}

The algorithm consists of three phases.
\begin{enumerate}
    \item In the first phase, we run $\IHT$ to obtain a warm start and a finer control on $\twonorm{\vths}$.
    \item In the second phase, we perform a sequence of adaptive queries (leveraging our access to Algorithm~\ref{alg:adap_est}) to identify a bounded-size superset of $\supp(\vths)$.
    \item In the third phase, once the support is recovered, we run simple OLS on the union of all estimated supports to obtain the desired $\ell_\infty$ estimation guarantee.
\end{enumerate}
The adaptive querying phase is the core of the algorithm and is built directly upon Lemma~\ref{lem:l2_contract}.
At a high level, this phase maintains a growing set $T_i$ of coordinates identified as heavy up to round $i$.
In subsequent rounds, the columns of the design matrix indexed by $T_i$ are muted, so that their contribution no longer interferes with future measurements.
Lemma~\ref{lem:l2_contract} ensures two key properties at each round:
(1) \emph{false discovery control}, which guarantees that the number of incorrectly identified coordinates is proportional to the number of remaining true coordinates, and hence that the final estimated support does not grow excessively;
and (2) \emph{\(\ell_2\)-energy contraction}, which guarantees that the remaining signal energy decreases geometrically across rounds. We use these two properties to show that the algorithm finds all signal coordinates in $O(\log k)$ adaptive rounds.
\begin{theorem}
\label{thm:partial-adaptivity}
Let $\delta \in (0, \half)$, $n = \Omega(k\log(k)\log (\frac d \delta))$ for an appropriate constant, and $k \in [d]$. Under Models~\ref{model:subg} and~\ref{model:for_all}, if Algorithm~\ref{alg:adap_est} is called with parameters
\[R \ge \twonorm{\vths},\quad r_2 < \sigma\text{ where } \sigma \defeq \max_{i \in \{0\} \cup [N+1]} \norm{\mx^{(i)\top}\vxi^{(i)}}_\infty,\quad r_\infty \ge \frac{C\sigma}{2},\quad N = \Omega(\log k),\]
then in the setting of Problem~\ref{prob:variable_selection}, with probability $\ge 1 - \delta$, the output of Algorithm~\ref{alg:adap_est} satisfies
    \begin{equation*}
        \supp(\vth) = \supp(\vths), \quad \infnorm{\vth - \vths} = O\Par{\sigma}.
    \end{equation*}
\end{theorem}
\begin{proof}
By Proposition~\ref{prop:RIP_sub_gaussian}, for large enough $n$, with probability $\ge 1 - \frac \delta 4$, $\mx^{(0)}, \mx^{(1)},\ldots,\mx^{(N)}$ all satisfy $(0.1, O(k))$-RIP, and $\mx^{(N)}$ satisfies $(0.1, O(k\log k))$-RIP, for appropriate constants. We condition on these events henceforth, and denote $S^\star \defeq \supp(\vths)$.
\paragraph{Phase I: Warm start.}
By Lemma~\ref{lem:IHT_l2_recover}, with probability at least \(1-\frac \delta 5 \), our $\IHT$ estimator
\(\vhth_0\) satisfies
\[
\twonorm{\vhth_0-\vths}
\le 10\sqrt k\sigma.
\]
Define the residual support $S_0 = S^\star\setminus T_0$.
Then
\[
\twonorm{\vths_{S_0}}
\le 
\twonorm{\vhth_0 -\vths}
\le 10\sqrt k \sigma.
\]
This warm start ensures that when entering the adaptive Phase II, the truncated signal $\vths_{S_0}$ satisfies
the \(\ell_2\)-norm condition required by Lemma~\ref{lem:l2_contract}. 

\paragraph{Phase II: Adaptive support identification.}
For each round \(i\ge 1\), let
$S_i = S^\star \setminus T_i$.
By construction of the algorithm, all columns \(\mx_{:T_i}\) are muted in future rounds.
Therefore, the observation \((\mx^{(i)},\vy^{(i)})\) is equivalent to the linear model
\begin{equation*}
\vy^{(i)} = \mx^{(i)} \vth^{(i)} + \vxi^{(i)}, \qquad \vth^{(i)} = \vths_{S_i}.
\end{equation*}
We apply Lemma~\ref{lem:l2_contract} to each round.
Since \(\min_{j\in S^\star}|\vths_j|\ge C\sigma\) under Problem~\ref{prob:variable_selection}, 
and $\norms{\vths_{S_i}}_2 \le \norms{\vths_{S_0}}_2 \le 10\sqrt k \sigma$,
the lemma is applicable at every iteration. Note that the conclusion still clearly holds if the truncation threshold is set larger than $\frac{C\sigma}{2}$, as both conclusions are monotone in the threshold.
Then, Lemma~\ref{lem:l2_contract} yields two uniform guarantees with probability $1- \frac \delta {4N}$.
\begin{enumerate}
    \item (\emph{False discovery control.}) The number of false inclusions at round \(i\) is at most \(k\).
    \item (\emph{\(\ell_2\)-energy contraction.})
    $\norms{\vths_{S_{i+1}}}_2 \le 
    0.95\norms{\vths_{S_i}}_2$.
    
\end{enumerate}
Here we note that the constant in the false discovery control property was selected by taking $C \ge 80$ and $C' = 10$ in \eqref{eq:false_count}.
By the assumption in Problem~\ref{prob:variable_selection}, we obtain that 
\[|S_i| \cdot C^2 \sigma^2 \le \twonorm{\vths_{S_i}}^2 \le 0.95^{2i} \cdot 100k\sigma^2,\]
which implies a geometric decay of support size: $|S_i| \le 0.95^{2i} \cdot \frac{100k}{C^2}$. Thus, after $N =  O(\log k)$ rounds, we have $|S_N| < 1$, i.e., $S_N = \emptyset$ and all true coordinates
have been identified. Moreover, the total size of the learned support $S \defeq T_N$ is at most $O(k\log k)$.

\paragraph{Phase III: OLS solution.}
Once a superset $S$ of the exact support $S^\star$ is recovered, we run ordinary least squares on the
identified support and then truncate. Denote $(\mx, \vxi) \defeq (\mx^{(N + 1)}, \vxi^{(N + 1)})$ for this part of the proof. We obtain that 
\begin{equation*}
    \vthp \defeq [\mx^\top\mx]_{S\times S}^{-1}\mx_{S:}^\top \vy = \vths + [\mx^\top\mx]_{S\times S}^{-1}\mx_{S:}^\top \vxi. 
\end{equation*}
By Lemma~\ref{lem:proper}, we know the top-$k$ truncation only affects the $\ell_\infty$ error by a factor of 2: 
\begin{equation*}
    \infnorm{\vth - \vths} = \infnorm{H_k(\vthp) - \vths} \le 2\infnorm{\vthp - \vths} = 2\infnorm{[\mx^\top\mx]_{S\times S}^{-1}\mx_{S:}^\top \vxi} = O\Par{\sigma},
\end{equation*}
with probability $\ge 1 - \frac \delta 4$, where the last inequality applied Corollary~\ref{cor:noise_linf_bound_indep} and leveraged our independence assumption. Taking a union bound over all failure events concludes the proof. 
\end{proof}

%% file: tex_file/append.tex
\section{Discussion of Error Metrics}
\label{append:erro_metric}

While our analysis naturally leads to bounds in terms of the quantity $\norms{\mx^\top \vxi}_\infty$, there are other common ways to parameterize the error of algorithms for Problem~\ref{prob:linfty_sparse}. This includes
\begin{equation}\label{eq:ols_error}\norm{\Brack{\mx^\top\mx}^{-1}_{S\times S} \mx^\top_{S:}\vxi}_\infty \text{ for } S \defeq \supp(\vths)\end{equation}
by e.g., \cite{wainwright2019high}, and
\[\frac{\norm{\vxi}_2}{\sqrt n},\quad \norm{\vxi}_\infty,\]
as sometimes seen in works from the heavy hitters literature.

In this section, we show that in the oblivious setting of Model~\ref{model:for_each}, these $\ell_\infty$-type error metrics are equivalent up to constant factors.
In contrast, under the adversarial setting of Model~\ref{model:adv_noise}, no such guarantee is achievable for other error metrics in general.

Lemma~\ref{lem:equiv_err} establishes conversions between these different error metrics in the oblivious setting.

\begin{lemma}
    \label{lem:equiv_err}
    Let $\delta \in (0,\half)$. In the context of Problem~\ref{prob:linfty_sparse} under Models~\ref{model:for_each} and~\ref{model:subg}, if $n =\Omega(k\log \frac d \delta)$ and $k = \Omega(\log \frac 1 \delta)$ for appropriate constants, then with probability $\ge 1 - \delta$, we have for $S \defeq \supp(\vths)$,\begin{align*}\infnorm{\Brack{\mx^\top\mx}^{-1}_{S\times S} \mx^\top_{S:}\vxi} &\stepa{=} \Theta\paren{\infnorm{\mx_{S:}^\top \vxi}} \stepb{=} O\paren{\twonorm{\vxi}\sqrt{\frac{\log \frac k \delta}{n}}} \stepc{=} O\Par{\infnorm{\vxi}\sqrt{\log \frac k \delta}}, \\
\infnorm{ \mx^\top_{S:}\vxi} &\stepd{=} \Omega\Par{\twonorm{\vxi}\sqrt{\frac{1}{n}}} \stepe{=} \Omega\Par{\infnorm{\vxi}\sqrt{\frac{1}{n}}}.
    \end{align*}
\end{lemma}
\begin{proof}
The identity $(a)$ follows directly from Lemma~\ref{lem:noise_linf_bound_fix}, and the bounds $(c)$, $(e)$ hold for all $\vxi \in \R^n$.

It remains to establish $(b)$ and $(d)$.
For any $i \in S$, conditioned on $\vxi$, we have
$\vx_i^\top \vxi \sim \subg(0, \frac C n \norms{\vxi}_2^2)$ with variance $\frac 1 n \norms{\vxi}_2^2$. The maximum of $k$ of these variables can be bounded by Lemma~\ref{lem:hoeffding} and a union bound, yielding $(b)$. Finally, there is a constant probability that each such draw $\vx_i^\top \vxi$ obeys the bound in $(d)$, by the same proof strategy as used in establishing \eqref{eq:anticonc_conclusion}. Under the stated parameter range on $k$, the probability $(d)$ fails under $k$ independent draws is at most $\delta$.

\end{proof}

We mention that both $(b)$ and $(d)$ can be tight: if $\vx_i$ is entrywise a (scaled) Gaussian, then $(b)$ is tight, and if $\vx_i$ is entrywise a (scaled) Rademacher, then for $1$-sparse $\vxi$, $(d)$ is tight. 

Regardless, Lemma~\ref{lem:linf_impossible_adv_1} establishes a complementary impossibility result: in the adaptive setting, no algorithm can guarantee an $\ell_\infty$ estimation error on the order of any of the other candidate metrics, even taking the more conservative of the pairs $(b)$, $(d)$ and $(a)$, $(e)$.

\begin{lemma}
\label{lem:linf_impossible_adv_1}
In the context of Problem~\ref{prob:linfty_sparse} under Models~\ref{model:adv_noise} and~\ref{model:subg}, denoting $S \defeq \supp(\vths)$, there is no algorithm outputting an estimate $\vth \in \R^d$ that can guarantee that with probability $\ge \frac 2 3$,
\begin{equation*} 
\norm{\vth - \vths}_\infty = O\paren{\infnorm{\vxi}}
\quad \text{or} \quad
O\paren{\twonorm{\vxi}\sqrt{\frac{\log k}{n}}} \quad \text{or} \quad O\Par{\infnorm{\Brack{\mx^\top\mx}^{-1}_{S\times S} \mx^\top_{S:}\vxi}}.
\end{equation*}
\end{lemma}

\begin{proof}
Consider two signal-noise pairs $(\vths_1, \vxi_1) = (\vzero_d, \mx \ve_i)$ and $(\vths_2, \vxi_2) = (\ve_i, \vzero_n)$, for any $i \in [d]$. In the third error metric, we take the $S$ in the first pair to be any set not containing $i$, and the $S$ in the second set to be $\{i\}$. Both induce the same observations: 
\begin{equation*}
\vy = \mX \vths_1 + \vxi_1
= \mX (\vth + \ve_j)
= \mX \vths_2 + \vxi_2,
\end{equation*}
and hence are statistically indistinguishable (i.e., no algorithm can beat random guessing). The parameters satisfy $\norms{\vths_1 - \vths_2}_\infty = 1$. On the other hand, by Lemma~\ref{lem:hoeffding}, standard $\chi^2$ concentration bounds (e.g., Lemma 1, \cite{LaurentM00}), and Corollary~\ref{cor:noise_linf_bound_indep} respectively, with probability at least $\frac 2 3$,
\begin{equation*}
\infnorm{\vxi_1} = \infnorm{\mX \ve_1}
= O\Par{\sqrt{\frac{\log n}{n}}},
\quad
\twonorm{\vxi_1} = O(1), \quad 
\infnorm{(\mx^\top\mx)^{-1}_{S\times S} \mx^\top_{S:}\vxi_1} = O\Par{\sqrt{\frac{\log d}{n}}},
\end{equation*}
and trivially $\vxi_2 = \vzero_n$ satisfies the same bounds. Note that in our application of Corollary~\ref{cor:noise_linf_bound_indep}, we used that $\vv \gets \mx \ve_i =  \vx_i$ is independent from $\mx_{:S}$ for $i \not\in S$.

Whenever the above bounds hold,
any of these error metrics allows for exact recovery of $\vths_1$ or $\vths_2$ when $n$ is large enough. This contradicts the maximum success probability of $\frac 2 3 \cdot \frac 1 2 + \frac 1 3$ achievable.
\end{proof}
\section{Tight Error Guarantee for Gaussian Noise Model}
\label{app:discuss_gaussian_noise}

In this section, we give an information-theoretic lower bound for $\ell_\infty$ sparse recovery under a Gaussian noise model. We start with a simple lemma to instantiate the general results in Sections~\ref{sec:for_each} and~\ref{sec:for_all}.

\begin{lemma}
    \label{lem:error_upper_bound}
    For any $\delta \in (0, \half)$, if $\mx \in \R^{n\times d}$ satisfies $(\half, 1)$-RIP, then for $\vxi \sim \Nor(0, \sigma^2 \id_n)$, with probability at least $1 - \delta$,
    \begin{equation*}
        \infnorm{\mx^\top \vxi} \leq O\paren{\sigma \sqrt{\log \frac{d}{\delta}}}.
    \end{equation*}
\end{lemma}
\begin{proof}
    For any $i\in[d]$, we have $\vx_i^\top \vxi \sim \calN(0, \sigma^2\twonorm{\vx_i}^2)$. Since $\twonorm{\vx_i}^2 \le \frac 3 2$, we have $\Pr[|\vx_i^\top \vxi| \ge t] \le 2\exp(-\tfrac{t^2}{2\sigma^2})$. Taking $t = O(\sigma \sqrt{\log d/\delta})$ and a union bound over $i\in[d]$ leads to the conclusion.
\end{proof}

We next quantify the intrinsic difficulty of $\ell_\infty$ sparse recovery using two standard notions of risk. The minimax risk captures the worst-case expected error over the $k$-sparse parameter class, corresponding to the partially-adaptive Model~\ref{model:for_all}, while the Bayes risk measures the average error under a prior on $\vths$, corresponding to the oblivious Model~\ref{model:for_each}. A basic but useful fact is that the Bayes risk under any prior lower bounds the minimax risk. Our proof roadmap mainly follows \cite{scarlett2019introductory}.

\begin{definition}
Let $\Theta$ be a parameter space such that every $\theta \in \Theta$ indexes a data distribution $\calP_\theta$.
Suppose we want to estimate an unknown $\theta \in \Theta$ from data $Z \sim \calP_\theta$. We define the minimax risk as 
\begin{equation*}
    \calM_{\minmax}(\Theta, Z) = \inf_{\vhth} \sup_{\vths \in \Theta} \E_{Z \sim \calP_{\vths}}\Brack{\infnorm{\vhth(Z) - \vths}},
\end{equation*}
where $\vhth$ is any estimate that is a function\footnote{We restrict to deterministic $\vhth$ without loss: otherwise outputting $\E \vhth$ improves the risk, since $\norm{\cdot}_\infty$ is convex.} of the observations $Z$.
Similarly, if $\theta$ is sampled from some prior $\pi$ over $\Theta$, we define the Bayes risk as 
\begin{equation*}
    \calM_{\bayes}(\pi, \Theta, Z) = \inf_{\vhth} \E_{\substack{\vths \sim \pi \\ Z \sim \calP_{\vths}}}\Brack{\infnorm{\vhth(Z) - \vths}}.
\end{equation*}
\end{definition}

Our lower bound follows a standard information-theoretic route. We construct a finite, well-separated subset of parameters and reduce estimation to identifying the correct candidate parameter. Applying Fano’s inequality relates the achievable estimation accuracy to the mutual information between the observations and the underlying index, yielding a quantitative lower bound on the Bayes (and therefore, minimax) risk. 

\begin{lemma}[Fano's inequality]
    \label{lem:fano_ineq}
    Let $V$ and $\hV$ be discrete random variables on a common set $\calV$, where $V \simu \calV$. Then letting $I$ denote the mutual information,
    \begin{equation*}
        \Pr\Brack{\hV \neq V} \geq 1 - \frac{I(V; \hV) + \log 2}{\log |\calV|}.
    \end{equation*}
\end{lemma}

The following calculation helps simplify applications of Lemma~\ref{lem:fano_ineq}.

\begin{lemma}[Lemma 4, \cite{scarlett2019introductory}]
    \label{lem:MI_kl_upper}
    let $P_V$, $P_{\vy}$ and $P_{\vy \mid V}$ be the marginal and conditional distributions with respect to a pair of jointly-distributed random variables $(V, \vy)$. Then for any auxiliary distribution $Q_\vy$, we have 
    \begin{equation*}
        I(V; \vy) = \sum_{v \in \calV} P_V(v) \kldiv{P_{\vy \mid V}(\cdot \mid v) \Vert P_\vy} \leq \sum_v P_V(v) \kldiv{P_{\vy \mid V}(\cdot \mid v) \Vert Q_\vy}.
    \end{equation*}
\end{lemma}

\begin{lemma}[Risk lower bound via exact recovery]
    \label{lem:risk_LB}
    Under Model~\ref{model:for_each}, fix $\eps > 0$, and let $\Theta_\calV = \{\vth_v\}_{v \in \calV}$ be a finite subset of $\Theta$, such that 
    \begin{equation}
        \label{eq: assum_packing}
        \infnorm{\vth_v - \vth_{v^\prime}} \geq \eps, \text{ for all } (v, v') \in \calV^2,\; v \neq v'.
    \end{equation}
    Then if $\pi$ is uniform over $\Theta_{\calV}$, and $\vy$ are observations generated from a distribution indexed by $\vth_V$, 
    \begin{equation*}
        \calM_{\minmax}(\Theta, \vy) \ge \calM_{\bayes}(\pi, \Theta, \vy) \geq \frac{\eps}{2}\paren{1 - \frac{I(V; \vy) + \log 2}{\log |V|}}.
    \end{equation*}
\end{lemma}
\begin{proof}
    By Markov's inequality, we have 
    \begin{equation*}
        \expect{\infnorm{\vhth - \vths}} \geq t\cdot \Pr\Brack{\infnorm{\vhth - \vths} > t}.
    \end{equation*}
    For any estimator $\vhth(\vy)$, let $\hV = \arg\min_{v\in \calV} \norms{\vhth - \vth_v}_\infty$. 
    Using the triangle inequality and our assumption~\eqref{eq: assum_packing}, if $\norms{\vhth - \vth_V}_\infty < \frac \eps 2$, then $\hV = V$, hence 
    \begin{equation*}
        \prob{\infnorm{\vhth - \vth_V} \geq \frac{\eps}{2}} \geq \prob{\hat{V} \neq V}.
    \end{equation*}
    Thus, taking $t \gets \frac \eps 2$ and applying Lemma~\ref{lem:fano_ineq} concludes the proof.
\end{proof}

\begin{lemma}
    \label{lem:info_LB}
    Under Model~\ref{model:for_each}, let $d = \Omega(k)$, $\sigma > 0$, $\eps \defeq \frac \sigma 2 \sqrt{\log \frac d k}$, and define 
    \[\Omega_k = \{\vths \in \R^d:  |\supp(\vths)| \le k\},\quad \calV = \Brace{\eps \vv : \vv \in \cbra{-1, 0, 1}^d, \onenorm{\vv} = k}.\] 
    Let $\pi$ be uniform over $\Theta_{\calV}$ where $\theta_{\vv} \defeq \vv$. If $\mx$ satisfies $(\half, 1)$-RIP and $\vxi \sim \calN(0, \sigma^2 \id_n)$, then
    \begin{equation*}
    \calM_{\minmax}(\Omega_k, \vy) \geq \calM_{\bayes}(\pi, \Omega_k, \vy) = \Omega\paren{\sigma\sqrt{\log \frac{d}{k}}}.
    \end{equation*}
\end{lemma}

\begin{proof}
Our construction of $\calV$ trivially satisfies the bound \eqref{eq: assum_packing}, so by Lemma~\ref{lem:risk_LB},
    \begin{equation}\label{eq:fano_conclusion}
        \calM_{\bayes}(\pi, \Omega_k, \vy) \geq \frac{\eps}{2}\paren{1 - \frac{I(V; \vy) + \log 2}{\log |\calV|}}.
    \end{equation}
    Similarly, by Lemma~\ref{lem:MI_kl_upper}, 
    \begin{align*}
    I(V; \vy) &\leq  \sum_{v \in \calV} P_V(v) \kldiv{P_{\vy \mid V}(\cdot \mid v) \Vert Q_{\vy}}.
    \end{align*}
    Taking $Q_{\vy} = \calN(0, \sigma^2)$ for all $\vy$, we have 
    \begin{align*}
    I(V; \vy) &\leq  \frac{1}{|\calV|}\sum_{v \in \calV}\frac{\twonorm{\mx\vth_v}^2}{2\sigma^2} = \frac{\eps^2}{2\sigma^2}\bbE_{V}\bra{\twonorm{\mx \vone_V}^2} = \frac{\eps^2}{2\sigma^2}\Tr(\mx \cov{V}\mx^\top).
    \end{align*}
    It is straightforward by a symmetry argument that $\cov{V} = \frac{k}{d}\id_d$, and thus 
    \begin{equation*}
        I(V; \vy) \leq \frac{\eps^2k}{2d\sigma^2}\normf{\mx}^2.
    \end{equation*}
    Since $\mx$ satisfies $(\half, 1)$-RIP, $\twonorm{\vx_i}^2 \leq 2$ for all $i \in [d]$, which implies $\normf{\mx}^2 \leq 2 d$. We then obtain 
    \begin{equation}
        \label{eq:mi_bound}
        I(V; \vy) \leq \frac{\eps^2k}{\sigma^2}.
    \end{equation}
    At this point, we can conclude the result by plugging \eqref{eq:mi_bound} back into \eqref{eq:fano_conclusion}:
    \[\frac{I(V; \vy) + \log 2}{\log |\calV|} \le \frac{k\log(\frac d k)}{2\log |\calV|} \le \half,\]
    where we used that the cardinality of $|\calV|$ satisfies
    \[|\calV| = \binom{d}{k} \cdot 2^k \ge \Par{\frac{2d}{k}}^k. \]
\end{proof}

Lemmas~\ref{lem:error_upper_bound} and~\ref{lem:info_LB} show that the minimax risk of $\ell_\infty$ sparse recovery, given an RIP matrix $\mx$ and Gaussian noise $\vxi$, even in the limit of infinite observations, is  $\Omega(\norms{\mx^\top \vxi}_\infty)$ with high probability. This is matched by all of our upper bounds up to (at most) a sublogarithmic overhead.

%% file: tex_file/appendix2.tex
\section{Application to Spike-and-Slab Posterior Sampling}\label{app:sample}

In this appendix, we briefly discuss one application of our upper bound results for $\ell_\infty$ sparse recovery. Recently, \cite{KumarSTZ25} gave an algorithm for Bayesian sparse linear regression in the following model.

\begin{model}[Spike-and-slab posterior sampling]\label{model:sample}
In the \emph{spike-and-slab posterior sampling} problem, $\mx \in \R^{n \times d}$ is drawn under Model~\ref{model:subg}, $\vxi \sim \Nor(0, \sigma^2)^n$ independently,  and $\vths \in \R^d$ is independently drawn from the following \emph{spike-and-slab} prior, for $\vq \in [0, 1]^d$ with $\norm{\vq}_1 = k$:
\[\pi \defeq \bigotimes_{i \in [d]} \Par{(1 - \vq_i) \delta_0 + \vq_i \Nor(0, 1)},\]
where $\delta_0$ is a Dirac density at $0$. For some $\delta \in (0, 1)$, the goal is to produce a sample from the posterior $\pi(\cdot \mid \mx, \vy)$ within total variation $\delta$, where $\vy = \mx \vths + \vxi$ are the observations.
\end{model}

Under Model~\ref{model:sample}, the parameter $k$ can be viewed as an expected sparsity level for the signal $\vths \sim \pi$, because coordinate $i \in [d]$ is only nonzero with probability $\delta$. The key challenge in solving Model~\ref{model:sample} is that for noise levels $\sigma$ that are $\approx 1$, there are many coordinates that could plausibly either be part of the signal $\vths$, or masked by the random noise $\vxi$. To model this uncertainty, the posterior density $\pi(\cdot \mid \mx, \vy)$ places nontrivial mass on $d^{\Omega(k)}$ candidate supports, and therefore the algorithm must successfully sample from this exponentially-sized candidate set.

Prior works by \cite{MC_YWJ16, montanari2024provably} respectively designed algorithms for restricted variants of Model~\ref{model:sample} that succeeded when $\sig$ is relatively large or small, and when the number of observations $n \gtrsim d$. The algorithm by \cite{KumarSTZ25} lifts these restrictions, and solves the sampling problem in Model~\ref{model:sample} for any $\sig > 0$ and using $n = \poly(k, \log(\frac d \delta))$ samples. In particular, \cite{KumarSTZ25} gives two algorithms: the first is based on $\ell_2$ sparse recovery, and thus runs in nearly-linear time (e.g., using Lemma~\ref{lem:IHT_l2_recover}), but uses $n \approx k^5$ observations to compensate for the weaker recovery guarantee. The second algorithm leverages $\ell_\infty$ sparse recovery, and uses an improved $n \approx k^3$, but its runtime was previously based on the LASSO (see Theorem 1, \cite{KumarSTZ25}, which claims a runtime of $\approx n^2 d^{1.5}$).

Substituting our Theorems~\ref{thm:l_inf_norm_bound_model1} or~\ref{thm:linf_iht} in place of the LASSO (Proposition 3, \cite{KumarSTZ25}) immediately gives the following improved runtime for spike-and-slab posterior sampling.

\begin{corollary}
In the setting of Model~\ref{model:sample}, if $n = \Omega(k^3 \polylog(\frac d \delta))$, there is an algorithm that returns $\vth \sim \pi'$ such that $D_{\textup{TV}}(\pi', \pi(\cdot \mid \mx, \vy)) \le \delta$ with probability $\ge 1 - \delta$ over the randomness of $(\mx, \vths, \vxi)$, which runs in time
\[O\Par{nd\log\Par{\frac{\log \frac 1 \delta}{\min(1, \sig)}}}.\]
\end{corollary}
\begin{proof}
The runtime of, e.g., Theorem~\ref{thm:linf_iht} meets the described bound, where we may take $r = \Omega(\sig)$ and $R^2 = O(\log \frac 1 \delta)$ with probability $\ge 1 - \delta$ under our modeling assumptions. The remainder of the proof follows identically to the proof of Theorem 1 in \cite{KumarSTZ25}.
\end{proof}